\documentclass[12pt,reqno]{amsart}
\usepackage{a4wide}
\numberwithin{equation}{section}

\usepackage{amsmath}
\usepackage{comment}
\usepackage{color}
\usepackage{txfonts}
\usepackage{amsfonts}
\usepackage{amsmath,amsthm,amssymb,amscd}
\usepackage{latexsym}
\usepackage{hyperref}
\usepackage[numbers,sort&compress]{natbib}
\usepackage{hypernat}
\allowdisplaybreaks

\numberwithin{equation}{section}

\newtheorem{theorem}{Theorem}[section]
\newtheorem{proposition}[theorem]{Proposition}

\newtheorem{lemma}[theorem]{Lemma}

\theoremstyle{definition}

\newtheorem{remark}[theorem]{Remark}

\newcommand{\va}{\varepsilon}

\newcommand{\la}{\lambda}

\def\r{\mathbb{R}}

\newcommand{\fenshi}{\frac{\alpha+\gamma-2\beta}{\alpha\gamma-\beta^2}}
\title{Circle-like concentrated solutions for two-component Bose-Einstein condensates}

\author{Qidong Guo \textsuperscript{1}, Qiaoqiao Hua \textsuperscript{2}, and Chongyang Tian  \textsuperscript{3}}

\address{ \textsuperscript{1}School of Mathematical Science, Anhui University, Hefei, 230601, P.R. CHINA}

\address{ \textsuperscript{2}School of Mathematical Sciences, University of Science and Technology of China, Hefei, 230026, P.R. CHINA}

\address{ \textsuperscript{3} School  of Mathematics and Statistics, Central  China Normal University, Wuhan 430079, P. R. China }

\email{qdguo@ahu.edu.cn}

\email{qqhua@ustc.edu.cn}

\email{cytian@mails.ccnu.edu.cn}

\begin{document}
	
\begin{abstract}
We investigate the normalized solutions of the following two-component Bose-Einstein condensates (BEC) system
\[
\left\{ 
\begin{aligned}
	-\Delta u + (\lambda+P(x))u &= \alpha u^3 +\beta uv^2, && \text{in } \mathbb{R}^2,\\
	-\Delta v + (\lambda+Q(x))v &= \gamma v^3 +\beta u^2 v, && \text{in } \mathbb{R}^2,
\end{aligned} 
\right.
\]
 with $L^2$-constraint $$\int_{\mathbb{R}^2}(u^2+v^2)\,dx = 1.$$
For any $\alpha>0$, $\gamma > 0$ and $\ \beta \in (-\sqrt{\alpha\gamma},0)\cup(0,\min \{\alpha,\gamma\})\cup \left(\max \{\alpha,\gamma\} , + \infty\right)$, we establish the existence of synchronized solutions concentrating on high-dimensional subsets of $\mathbb{R}^2$ by employing a finite-dimensional reduction method combined with some local Pohozaev identities. More precisely, we construct vector radial solutions that concentrate on circles when
$ \frac{\alpha + \gamma - 2\beta}{\alpha\gamma - \beta^2}$ tends to zero.
Our results fill the blank in the system for high-dimensional concentrated normalized solutions.\\
		
		
	\end{abstract}
		\maketitle
	\section{Introduction}

In this paper, we	investigate the existence of normalized solutions to the following coupled nonlinear Schr\"odinger system
	\begin{equation}\label{eqmain}
		\begin{cases}
			-\Delta u + (\la+P(x)) u= \alpha u^3 + \beta  uv^2,\  &\text{ in } \ \r^{2},\\
			-\Delta v +( \la+Q(x)) v= \gamma v^3 + \beta  u^2v,\  &\text{ in } \ \r^{2},\\
            \int_{\r^2}(u^2+v^2)dx = 1,  
		\end{cases} 
	\end{equation}
	with $\alpha>0$, $\gamma>0$, $\beta\in\mathbb{R}$ and $\lambda$ being a Lagrangian multiplier.

It is well-known that the system \eqref{eqmain} arises in several important physical contexts. It serves as a model for Bose-Einstein condensates (BEC) with two interacting components, where $P(x)$ and $Q(x)$  represent the external potentials, see \cite{D.S., T. Lin, G.M, C.J.}. The parameters $\alpha>0$ and $\gamma>0$ indicate that the intraspecies interaction among atoms within each component is attractive, while $\beta\in\r$  represents the interspecies interaction between components. The sign of $\beta$ determines whether the interaction is attractive ($\beta<0$ ) or repulsive ($\beta>0$) \cite{Y. Guo}. Due to the interspecies interaction between the two components, two-component Bose-Einstein condensates exhibit more complex phenomena than single-component Bose-Einstein condensates. These additional interaction effects introduce new phenomena and nontrivial challenges, further driving the need for comprehensive and systematic studies of two-component Bose–Einstein condensates.

The groundbreaking experiments that first observed Bose-Einstein condensation in dilute alkali metal vapors were successfully achieved in 1995. With these experimental breakthroughs, mathematicians' interest in the Gross-Pitaevskii (GP) framework \cite{E. Gross}
\begin{equation}\label{eq:GP_time}
i\partial_t \psi(x,t) = -\Delta \psi(x,t) + V(x)\psi(x,t) - a|\psi(x,t)|^2\psi(x,t), \quad x \in \mathbb{R}^2,
\end{equation}
subject to the normalization condition $\int_{\mathbb{R}^2}|\psi(x,t)|^2\,dx = 1$, proposed by Gross and Pitaevskii in the 1960s, was reignited.  Seeking solutions to system \eqref{eq:GP_time} of the form $\psi(x,t) = u(x)e^{-i\mu t}$, where $\mu$ denotes the chemical potential and $u(x)$ is a time-independent function, reduces the system \eqref{eq:GP_time} to the following nonlinear problem about $(u,\mu)$,
\begin{equation}\label{eq:stationary}
-\Delta u + V(x)u = au^3 + \mu u, \quad \text{in } \mathbb{R}^2,
\end{equation}
with the constraint
\begin{equation}\label{eq:normalization}
\int_{\mathbb{R}^2} u^2 dx= 1,
\end{equation}
which is a single-component Bose-Einstein condensate of system \eqref{eqmain} if we take $\la=-\mu$.
A standard variational approach to construct solutions of \eqref{eq:stationary}--\eqref{eq:normalization} involves minimizing the associated energy functional
\begin{equation*}
\label{eq:minimization}
I_a := \inf\left\{ \frac{1}{2}\int_{\mathbb{R}^2}\left(|\nabla u|^2 + V(x)u^2\right) dx - \frac{a}{4}\int_{\mathbb{R}^2}u^4dx : u \in H^1(\mathbb{R}^2), \int_{\mathbb{R}^2}u^2 dx = 1 \right\}.
\end{equation*}
 The asymptotic behaviors and local uniqueness of ground states to \eqref{eq:stationary}--\eqref{eq:normalization} have been studied in \cite{YGCL,YGRS,YGXZ} and the references therein. Any solutions with energy strictly greater than ground states are often called excited states which have been obtained by constructing multi-peak solutions in \cite{GHLY, LPWY} and radial solutions concentrating on spheres in \cite{CVPDE}, respectively. Moreover, Guo and Tian in \cite{JDE} prove the existence of solutions to the following general Schr\"odinger equation 
\begin{equation*}
-\Delta u + V(x)u = au^p + \mu u, \quad \text{in } \mathbb{R}^n,
\end{equation*}
with the constraint
\begin{equation*}
\int_{\mathbb{R}^n} u^2 dx= 1,
\end{equation*} where $n\geq2$ and $p>1$. 
Denoting $\epsilon=\frac{1}{\sqrt{-\mu}}$,
the solutions obtained in \cite{JDE} concentrate on the sphere $|x|=t_\epsilon$, with $t_\epsilon$ being the non-degenerate critical point of 
\[M_\epsilon(r):=\epsilon^{2(n-2)}r^{n-1}(1+\epsilon^2V(r))^{\frac{p+3}{2(p-1)}}.\]

Problems with the fractional Laplacian have been studied in \cite{GQWCH}. Guo, Wang and Yang have established results concerning the existence and local uniqueness of normalized $k$-peak solutions to the following fractional nonlinear Schr\"odinger equation
\[
\begin{cases}
(-\Delta)^s u + V(x)u = a u^p + \mu u, & x \in \mathbb{R}^N, \\
u \in H^s(\mathbb{R}^N),
\end{cases}
\]
under the mass constraint
\[
\int_{\mathbb{R}^N} u^2(x)\,dx = 1,
\]
where \(s \in (0,1)\), \(p \in (1, 2_s^* - 1)\), and \(2_s^* = \frac{2N}{N-2s}\) denotes the fractional critical Sobolev exponent.
For more concentration phenomena about single-component BEC or without $L^2$-constraint, we can refer to \cite{CMP,CPAM,KT,Duke,MY} and the references therein.

What is more, the existence, uniqueness and asymptotic behaviors of the ground state to system \eqref{eqmain} are proved in \cite{YGSLJWXZ,YGSLJW} by variational methods. We denote \begin{equation}
    \label{a***}
a^*=\int_{\r^2}w^2 dx,
\end{equation}
where $w$ is the unique positive radial solution of $-\Delta w+w=w^3$, in $\r^2$. Guo and Yang \cite{JDE1} use the finite dimensional reduction method to obtain the existence of peak solutions. In precise, they proved that,
as $\frac{\alpha+\gamma-2\beta}{\alpha\gamma-\beta^2}\to\frac{1}{ka^*}$, system \eqref{eqmain} has $k$-peak solutions concentrating on $k$ common critical points of $P(x)$ and $Q(x)$. However, the existence of high-dimensional concentrated solutions for system \eqref{eqmain} remains an open problem. 

In this paper, we focus on the following two questions:

\textbf{Question 1.} If problem \eqref{eqmain} exists a high-dimensional concentrated solution, then what necessary conditions will parameters $\alpha,\beta,\gamma$ and potentials $P,Q$ satisfy?

\textbf{Question 2.} Under above suitable conditions, does there exists a high-dimensional concentrated solution to problem \eqref{eqmain}?

	To analyze \textbf{Question 1}, we first consider the simple case $P(x)=Q(x)=0$ and introduce 
	\begin{equation}
	    \label{va}
\va:=\frac{\alpha+\gamma-2\beta}{\alpha\gamma-\beta^2}.
	\end{equation}
	Let $u(x)=(\frac{\lambda(\gamma-\beta)}{\alpha\gamma-\beta^2})^{\frac{1}{2}}w( \lambda^{\frac{1}{2}} x)$ and $v(x)=(\frac{\lambda(\alpha-\beta)}{\alpha\gamma-\beta^2})^{\frac{1}{2}}w(\lambda^{\frac{1}{2}} x)$. Then $(u,v)$ satisfies the following sysytem
	\[
	\begin{cases}
		-\Delta u + \la u= \alpha u^{3} + \beta uv^{2},\  &\text{ in } \ \r^{2},\\
		-\Delta v + \la v= \gamma v^{3} + \beta u^{2}v,\  &\text{ in } \ \r^{2},\\
	\end{cases} 
	\]
	and
	\[
	\int_{\mathbb R^2} u^2 dx=
	\frac{\gamma-\beta}{\alpha\gamma-\beta^2}\int_{\mathbb R^2} w^2 dx,
	\]
	\[
	\int_{\mathbb R^2} v^2 dx= \frac{\alpha-\beta}{\alpha\gamma-\beta^2}\int_{\mathbb R^N}w^2 dx.
	\]
	By the constraint in \eqref{eqmain}, we see that 
	\[
	\va \int_{\mathbb R^2} w^2 dx =1.
	\]
	Recall $a^*=\int_{\mathbb R^2} w^2 dx$ defined in \eqref{a***}. This shows that as $\lambda\to+\infty$, there must be $\va\to\frac{1}{a^*}$. Similarly, if one considers $k$-peak solutions, then there will be $\va\to\frac{1}{ka^*}$. This is consistent with the conclusion in \cite{JDE1}. In this paper, we are devoted to find a solution concentrating on a high dimensional subset of $\mathbb{R}^2$. Intuitively, the number of peaks is infinite. So, the parameter $\va$ should be close to $0$.

		Next, we consider non-trivial potential functions $P(x), Q(x)$. If  system \eqref{eqmain} possesses a radially symmetric solution concentrated on a circle, what necessary conditions should the potential functions $P(x), Q(x)$ satisfy? Inspired by the literature \cite{CMP}, firstly, the potentials should be radial functions, i.e., $P(x) = P(|x|), Q(x) = Q(|x|) $. Secondly, the function determined by $P(x), Q(x) $
		\begin{equation}
			\label{Mdingyi}
			M_\la(r):=r\left[1+\frac{1}{\la}\Big(\frac{\gamma-\beta}{\alpha+\gamma-2\beta} P(r) + \frac{\alpha-\beta}{\alpha+\gamma-2\beta} Q(r) \Big)\right]^{\frac{3}{2}}
		\end{equation}
		should have a sequence of critical points $r_{\lambda}$. Intuitively, this function is a generalization of the function
		\[
		M(r) = r^{N-1} V^{\sigma}(r)
		\]
		 in the literature \cite{CMP}, where $N$ is the spatial dimension, $\sigma = \frac{p+1}{p-1} - \frac{1}{2} $, and $p > 1$ is the highest exponent of the nonlinear terms.
		At this moment, we basically answer \textbf{Question 1}. We will state the detailed conclusions in Theorem \ref{diyigedingli} and provide a rigorous proof in Section \ref{biyaoxing}.
		
		Now let us introduce some notations. Throughout this paper, we define
		a norm in $H^{1}(\mathbb{R})\times H^{1}(\mathbb{R})$ by
		\[ \|(u,v)\|_{\la}^2 :=  \|u\|_{1,\la}^2 + \|v\|_{2,\la}^2 , \]
		where
		\[ \|u\|_{1,\la} = \Bigl(\int_{0}^{+\infty}r \bigl(|u'|^{2} +(\la+P(r)) u^{2} \bigr) dr \Bigr)^{\frac{1}{2}}, \quad
		\|v\|_{2,\la} = \Bigl(\int_{0}^{+\infty}r \bigl(|v'|^{2} +(\la+Q(r)) v^{2} \bigr) dr \Bigr)^{\frac{1}{2}},
		\]
		endowed with the inner product
		\[ \langle (u, v), (\phi,\psi)\rangle_{\la} =\int_{0}^{+\infty}r  \bigl( u'  \phi' +  v'  \psi' +(\la+P(r)) u \phi +(\la+Q(r)) v \psi \bigr) dr, \ \forall u, v, \phi, \psi \in H^{1}(\mathbb{R}).\]
		Let 
		\begin{equation*}
			\begin{array}{l}
				U_{r_{\la},\la}(r)= \sqrt{\la +P(r_\la)}
				U(\sqrt{\lambda+P(r_\la)} (r - r_{\la})),\\
				V_{r_{\la},\la}(r)= \sqrt{\la +Q(r_\la)}
				V(\sqrt{\lambda+Q(r_\la)} (r - r_{\la})),
			\end{array}
		\end{equation*} 
		where
		\begin{equation}\label{sol_lim}
			(U,V) =  \Bigg( \sqrt{\frac{\gamma-\beta}{\alpha\gamma-\beta^2}} w,  
			\sqrt{\frac{\alpha-\beta}{\alpha\gamma-\beta^2}} w
			\Bigg),
		\end{equation}
		is the unique radial positive solution of system
		\begin{equation}\label{eqlim}
			\begin{cases}
				- u'' +  u= \alpha u^{3} + \beta uv^{2},\  &\text{ in } \ \r,\\
				- v'' +  v= \gamma v^{3} + \beta u^{2}v,\  &\text{ in } \ \r,\\
			\end{cases} 
		\end{equation}
		with $\beta \in (-\sqrt{\alpha\gamma},0)\cup(0,\min\{\alpha,\gamma\})\cup (\max\{\alpha,\gamma\},+\infty)$ and 
		$w\in H^1(\r)$ is the unique radial positive solution of 
		\begin{equation}\label{w}
			- w''+ w=  w^{3}, \  \text{ in } \ \r. 
		\end{equation}
		Then $(U_{r_{\la},\la},V_{r_{\la},\la})$ satisfy the following system
		\begin{equation}\label{eqmain_approxi}
			\begin{cases}
				- u'' + (\la+P(r_\la)) u= \alpha u^3 + \beta  uv^2,\  &\text{ in } \ \r,\\
				- v'' +( \la+Q(r_\la)) v= \gamma v^3 + \beta  u^2v,\  &\text{ in } \ \r.\\
			\end{cases} 
		\end{equation}

		Our first result states that
		\begin{theorem}
			\label{diyigedingli}
			Let $\alpha>0$, $\gamma>0$, $\beta \in (-\sqrt{\alpha\gamma},0)\cup(0,\min \{\alpha,\gamma\})\cup \left(\max \{\alpha,\gamma\} , + \infty\right)$, $P(x)=P(|x|)$ and $Q(x)=Q(|x|)$ are $C^1$ functions with $P,Q,P^\prime,Q^\prime$ being bounded. Suppose 
			 $(u_\la,v_\la)$ is a concentrated radial solution of \eqref{eqmain} with the form
			\begin{equation}
				\label{tongshi}
				(u_\la(r),v_\la(r))=\big(U_{r_{\la},\la}(r),V_{r_{\la},\la}(r)\big)+(\omega_{1,\la}(r), \omega_{2,\la}(r)),
			\end{equation}
			and 
			\begin{equation}
				\label{yuxiang}\|(\omega_{1,\la}(r), \omega_{2,\la}(r))\|_\la=o\Big(\|\big(U_{r_{\la},\la}(r),V_{r_{\la},\la}(r)\big)\|_\la\Big)=o(\la^{\frac{3}{4}}r_{\la}^{\frac{1}{2}}).
			\end{equation}
			 Then it must holds that
			\[\va=\fenshi \to0,\,\,r_\la\to+\infty.\]
			Moreover, $r_\la$ satisfies
			\[M_\la^\prime(r_\la)=o(1)\,\,\text{or}\,\, (\gamma-\beta)P^\prime(r_\la)+(\alpha-\beta)Q^\prime(r_\la)=o(1).\]
		\end{theorem}
\begin{remark}
	In \eqref{tongshi}, $r_\lambda\in\mathbb{R}$ is the common maximum point of $u_\lambda$ and $v_\lambda$. We call these solutions \textit{synchronized solutions}. On the other hand, if the solutions concentrate at different points, we call them \textit{segregated solutions}, see \cite{QGQH, GQZCX, WLSP} for example.
\end{remark}

	As for \textbf{Question 2},  what we need to answer is that whether the solution will exhibit high-dimensional concentration under suitable conditions.
	In order to prove the existence of solutions to system \eqref{eqmain}, we will apply the Lyapunov-Schmidt reduction method.  However, there are two key issues that need to be resolved. One is to show the invertibility of the associated linearized operator. The other is to prove the solvability of the Lagrangian multiplier $\la$.
	The key to deal with the first issue lies in the study of the non-degeneracy of solutions to the limit problem \eqref{eqlim}.
	In terms of the high dimensional case
		\begin{equation}\label{ssss}
		\begin{cases}
			-\Delta u+  u= \alpha u^{3} + \beta uv^{2},\  &\text{ in } \ \r^N,\\
			- \Delta v+  v= \gamma v^{3} + \beta u^{2}v,\  &\text{ in } \ \r^N,\\
		\end{cases} 
	\end{equation}
	one can refer to  \cite{Ambrosetti2006,Bartsch2006,T. Lin,Sirakov2007} for the existence of ground states and bound states. In \cite{Wei2012}, J. Wei and W. Yao 
	proved that the solutions to system \eqref{ssss} are unique for sufficiently small $\beta > 0$, 
	while the solution of \eqref{ssss} is a unique positive radial solution for $\beta > \max\{\alpha, \gamma\}$.
For $N\leq3$, Dancer and Wei studied the spike solutions of the system \eqref{ssss}
with perturbation of the terms $\Delta u, \Delta v$ and gave a nondegenerate result with $\beta>0$ (refer to Theorem 1.3 in \cite{Dancer2009}). 
Later, Peng and Wang  in \cite{SPZW} improved this nondegenerate result for a wider range of $\beta$ when $N=3$.
We similarly rewrite the nondegenerate result about \eqref{sol_lim} as follows. 
\begin{lemma}[cf. Proposition 2.3 in \cite{SPZW}] \label{beta}
	There exists a decreasing sequence
	$\{\beta_k\}  \subset (-\sqrt{\alpha\gamma}, 0) $ with
	$\beta_k\to -\sqrt{\alpha\gamma}$ as $k\to\infty$ such that for $\beta\in (-\sqrt{\alpha\gamma}, 0) \cup (0,\min\{\alpha,\gamma\}) \cup
	(\max\{\alpha,\gamma\}, \infty)$ and $\beta \neq \beta_k$ for any $k$, $(U, V )$ is non-degenerate for the system \eqref{eqlim}
	in $H^1(\r)\times H^1(\r)$ in the sense that the kernel is given by span$\{(\theta(\beta)\frac{\partial W}{\partial r} , \frac{\partial W}{\partial r} )\}$, where $\theta(\beta) \neq 0$.
\end{lemma}


 \noindent The second issue is to show that there exists a $\lambda_{\varepsilon}$ tending to $+\infty$ for any $ \varepsilon$ close to $0$. Based on many research background on normalization problems, such as \cite{GHLY,QGQH}, we only need to obtain the relationship between $\lambda$ and $\varepsilon$ through some analytical tools such as Pohozaev identities, and then we can prove that the system \eqref{eqmain} has radial solutions concentrated at $ r_{\varepsilon}$ which tends to $+\infty $. We will solve the relationship between $\lambda$ and $\varepsilon$ by calculating the $L^2$-energy of the constructed solution (see Section \ref{cz2}). Then the vector solution $(u_\la,v_\la)$ can be rewritten as $(u_\va,v_\va)$. Precisely,
 \begin{equation}
 	\label{tongshi1}
 	(u_\va(r),v_\va(r))=\big(U_{r_{\va},\la_\va}(r),V_{r_{\va},\la_\va}(r)\big)+\omega_\va,
 \end{equation}
 and 
 \begin{equation}
 	\label{yuxiang1}\|\omega_\va\|_{\la_\va}=o\Big(\|\big(U_{r_{\va},\la_\va}(r),V_{r_{\va},\la_\va}(r)\big)\|_{\la_\va}\Big)=o(\la_\va^{\frac{3}{4}}r_{\va}^{\frac{1}{2}}),
 \end{equation}
 where $r_\va:=r_{\la_\va}$, $\omega_\va(r):=(\omega_{1,\va}(r), \omega_{2,\va}(r)):= (\omega_{1,\la_\va}(r), \omega_{2,\la_\va}(r))$.

In order to construct concentrated solutions to system \eqref{eqmain} when $\va>0$ is sufficiently small, we define 
\begin{equation}
    \label{Mdingyi1}
M_\va(r):=M_{\la_\va}(r) = r\left[1+\frac{1}{\la_\va}\Big(\frac{\gamma-\beta}{\alpha+\gamma-2\beta} P(r) + \frac{\alpha-\beta}{\alpha+\gamma-2\beta} Q(r) \Big)\right]^{\frac{3}{2}},
\end{equation}
and assume that the function $M_\va(r)$ has a sequence of non-degenerate critical points which tends to infinity (see \cite{CVPDE,JDE,CMP}). Namely, there exist a sequence $\{y_\va\}\subset \r^+$ and some positive constants $C_0$, $C_1$ and $C_2$ such that 
\begin{equation}
    \label{M1}
    M_\va^\prime(y_\va)=0\,\,and \,\,|M_\va^{\prime\prime}(y_\va)|>C_0, \text{ for all } \va\leq \va_0,
\end{equation}
and
\begin{equation}
    \label{M2}
   C_1\la_\va \leq y_\va\leq C_2\la_\va,
\end{equation}
where $y_\va$ denotes $y_{\la_\va}$.

We answer \textbf{Question 2} by proving the following theorem.
\begin{theorem}
\label{diergedingli}
Let $\alpha>0$, $\gamma>0$, $\beta$ satisfy the condition in Lemma \ref{beta}, $P(x)=P(|x|)$ and $Q(x)=Q(|x|)$ are $C^1$ functions with $P,Q,P^\prime,Q^\prime$ being bounded. If $M_\va(r)$ satisfies \eqref{M1} and \eqref{M2}, then there exists an $\va_0>0$ such that for any $\va\in(0,\va_0]$, problem \eqref{eqmain} has a vector radial solution $(u_{\varepsilon}, v_{\varepsilon})$ of the form \eqref{tongshi1}-\eqref{yuxiang1} satisfying that
\[
\la_\va \in \left(\left( 8A  \pi \varepsilon\right)^{-\frac{1}{\frac{3}{2} + \theta}}, \left( 4A  \pi \varepsilon\right)^{-\frac{1}{\frac{3}{2} - \theta}} \right) , \ 
|r_\va-y_\va|=O\Big(\la_{\va}^{-\frac{1}{2}+\theta} \Big) , 
\text{ and } \|\omega_\va\|_{\la_\va} = O\Big( \la_{\va}^{-\frac{1}{4}+\frac{\theta}{2}}\Big),
\]
where $\theta>0$ is a small constant.
\end{theorem}

\begin{remark}
    As in \cite{JDE} or \cite{CVPDE}, if we take $P(r)=Q(r)=\sin r$, then $M_\va(r)$ satisfies \eqref{M1} and \eqref{M2}. In this case, $M_\va(r)$ admits multiple distinct non-degenerate critical points provided that $\va$ is sufficiently small. This yields the existence of both multi-peak solutions and multiple single-peak solutions. Furthermore, if we suppose $N(\va)$ denotes the number of single-peak solutions to \eqref{eqmain}, then $N(\va)$ is a non-decreasing function and as $\va\to 0$, $N(\va)\to+\infty$.
\end{remark}

\begin{remark}
	To the best of our knowledge, it seems to be the first result on high-dimensional concentrated solutions for nonlinear Schr\"{o}dinger systems which we obtain via the construction of radial solutions concentrated on a circle. Furthermore, inspired by the recent work of \cite{Duan}, we anticipate that non-radial solutions concentrated on general closed curves can also be constructed.
\end{remark}

This paper is organized as follows. In Section \ref{biyaoxing}, we will make full use of various identities and estimates to prove Theorem \ref{diyigedingli}. In Section \ref{cz1} we investigate the existence of solutions to system \eqref{eqmain} without $L^2$-constraint, and in Section \ref{cz2} we are devoted to establish the relationship between $\lambda$ and $\varepsilon$ and give the proof of Theorem \ref{diergedingli}.

    \section{Necessity }\label{biyaoxing}
	In this section, we will find the necessary conditions that determine the concentration property of radial solutions to problem \eqref{eqmain}. At the starting of this section, we 
	deduce some identities with respect to $w(r), U(r)$ and $V(r)$.
	\begin{lemma}\label{pohozaev}
		It holds
		\begin{equation}\label{pohozaev_single}
			3 \int_0^{+\infty} w^4(r)dr = 4 \int_0^{+\infty} w^2(r)dr = 12 \int_0^{+\infty} |w'(r)|^2 dr.
		\end{equation}
	\end{lemma}
	
	\begin{proof}
		Since $w(r) \in H^1(\mathbb{R})$ solves \eqref{w},
		multiplying $w$ on both sides of  \eqref{w} and integrating over $(0, +\infty)$, we get
		\[
		\int_0^{+\infty} |w'(r)|^2 dr + \int_0^{+\infty} w^2(r) dr = \int_0^{+\infty} w^4(r) dr.
		\]
		Similarly, multiplying $rw'$ on both sides of  \eqref{w} and integrating over the same interval, we have
		\[
		- \int_0^{+\infty} |w'(r)|^2 dr + \int_0^{+\infty} w^2(r) dr = \frac{1}{2} \int_0^{+\infty} w^4(r) dr.
		\]
		Hence \eqref{pohozaev_single} holds by solving above systems.
	\end{proof}
	\begin{lemma}\label{pohozaev_system}
		It follows that
		\begin{equation}\label{pohozaev1}
			\int_{0}^{+\infty}(U^{2}+V^{2})dr = \frac{3}{4}\int_{0}^{+\infty} (\alpha U^{4} + \gamma V^4 +2\beta U^{2}V^{2}) dr ,
		\end{equation}
	and
	\begin{equation}\label{pohozaev2}
		\int_{0}^{+\infty}(U^{2}+V^{2})dr = 3 \int_{0}^{+\infty}(|U'|^{2}+|V'|^{2})dr.
	\end{equation}
	\end{lemma}
	
	\begin{proof}
		By equation \eqref{eqlim}, $(U(r),V(r)) \in H^{1}(\mathbb{R})\times H^{1}(\mathbb{R})$ satisfies 
		\begin{equation}\label{2-0}
			\begin{cases}
				- U'' +  U= \alpha U^{3} + \beta UV^{2},\  &\text{ in } \ \r,\\
				- V'' +  V= \gamma V^{3} + \beta U^{2}V,\  &\text{ in } \ \r.\\
			\end{cases}
		\end{equation}
		Multiplying $U$ on both sides of $\eqref{2-0}_1$ and integrating over $(0,+\infty)$, 
        we thus obtain that
		\begin{equation}\label{2-1}
			\int_{0}^{+\infty}(|U'|^{2}+U^{2})dr 
			= \int_{0}^{+\infty} (\alpha U^{4} +\beta U^{2}V^{2}) dr. 
		\end{equation}
	Similarly, multiplying $V$ on both sides of $\eqref{2-0}_2$ and integrating over $(0,+\infty)$, one can show that
	\begin{equation}\label{2-2}
		\int_{0}^{+\infty}(|V'|^{2}+V^{2})dr 
		= \int_{0}^{+\infty} (\gamma V^{4} +\beta U^{2}V^{2}) dr. 
	\end{equation}
From \eqref{2-1} and \eqref{2-2}, we have
\begin{equation}\label{2-3}
	\int_{0}^{+\infty}(|U'|^{2}+|V'|^{2})dr +\int_{0}^{+\infty}(U^{2}+V^{2})dr 
	= \int_{0}^{+\infty} (\alpha U^{4} + \gamma V^4 +2\beta U^{2}V^{2}) dr. 
\end{equation}

By a similar argument, multiplying $rU'$ and $rV'$ on both sides of $\eqref{2-0}_1$ and $\eqref{2-0}_2$ respectively, then integrating over $(0,+\infty)$ yields that
\begin{equation}\label{2-4}
	-\int_{0}^{+\infty}|U'|^{2} dr + \int_{0}^{+\infty}U^{2}dr 
	= \frac{1}{2} \int_{0}^{+\infty}\alpha U^4 dr -2 \int_{0}^{+\infty} \beta r UU'V^2 dr,
\end{equation}
and
\begin{equation}\label{2-5}
	-\int_{0}^{+\infty}|V'|^{2} dr + \int_{0}^{+\infty}V^{2}dr 
	= \frac{1}{2} \int_{0}^{+\infty}\gamma V^4 dr -2 \int_{0}^{+\infty} \beta r U^2VV' dr.
\end{equation}
Therefore, equations \eqref{2-4} and \eqref{2-5} give rise to
\begin{equation}\label{2-6}
	\begin{split}
		-\int_{0}^{+\infty}(|U'|^{2} +|V'|^{2}) dr + \int_{0}^{+\infty}(U^{2} +V^{2})dr 
		&= \frac{1}{2} \int_{0}^{+\infty}(\alpha U^4 +\gamma V^4) dr - \int_{0}^{+\infty} \beta r (U^2V^2)' dr \\
		&= \frac{1}{2} \int_{0}^{+\infty}(\alpha U^4 +\gamma V^4) dr + \int_{0}^{+\infty} \beta U^2V^2 dr.
	\end{split}
\end{equation}
Combining \eqref{2-3} and \eqref{2-6}, we obtain \eqref{pohozaev1} and \eqref{pohozaev2}.
	\end{proof}
	
	Firstly, we consider the radial solution of problem \eqref{eqmain} without $L^2$-constraint:
	\begin{equation}\label{eqmain_without}
			\begin{cases}
					-\Delta u + (\la+P(|x|)) u= \alpha u^3 + \beta  uv^2,\  &\text{ in } \ \r^{2},\\
				-\Delta v +( \la+Q(|x|)) v= \gamma v^3 + \beta  u^2v,\  &\text{ in } \ \r^{2}.\\
				\end{cases} 
		\end{equation}
	Recall that $(U,V) \in H^{1}(\mathbb{R})\times H^{1}(\mathbb{R})$ is the radial positive solution of problem \eqref{eqlim}. Assume that $(u_{\la}, v_{\la})$ is a radial solution of problem \eqref{eqmain_without} with the form
	\begin{equation}\label{u_la}
		(u_\la(r),v_\la(r))=(U_{r_{\la},\la}(r),V_{r_{\la},\la}(r))+(\omega_{1,\la}(r), \omega_{2,\la}(r)),
	\end{equation}
where $$U_{r_{\la},\la}(r)= \sqrt{\la +P(r_\la)} U(\sqrt{\lambda+P(r_\la)} (r - r_{\la})),$$
$$V_{r_{\la},\la}(r)= \sqrt{\la +Q(r_\la)} V(\sqrt{\lambda+Q(r_\la)} (r - r_{\la}))$$ and 
\begin{equation}
    \label{yx}\|(\omega_{1,\la}(r), \omega_{2,\la}(r))\|_\la=o(\|(U_{r_{\la},\la}(r),V_{r_{\la},\la}(r))\|_\la)=o(\la^{\frac{3}{4}}r_{\la}^{\frac{1}{2}}).
\end{equation}

	Now we give the a prior estimate of $(u_{\la},v_\la)$, which is similar to Lemma 2.1.2 in \cite{CPY}. We omit the proof.
	
	\begin{lemma}\label{jieyoujie}
		Suppose that $(u_{\la},v_{\la})$ is a positive radial solution of form \eqref{u_la}. Then there exist $\eta > 0$ and $C > 0$, such that
		\begin{equation}\label{decay1}
			u_{\la}(x), v_{\la}(x)\leq C\sqrt{\lambda} e^{-\eta \sqrt{\lambda} \,|\, |x| - |x_{\la}| \,| }, \quad \text{for any } x\in \mathbb{R}^{2} . 
		\end{equation}
	Moreover, we have
	\begin{equation}\label{decay2}
	|u'(x)|,|v'(x)| \leq C e^{-\frac{1}{4}\eta \sqrt{\lambda} \delta} \ \text{ with } |x|=r, |x_{\la}| = r_{\la} \text{ and } |r-r_\la|=\delta.
	\end{equation}
	\end{lemma}

	Using Lemma \ref{jieyoujie}, we can obtain the following estimate on the error term $(\omega_{1,\la},\omega_{2,\la})$. The proof is similar to the proof of Lemma 2.3 in \cite{CVPDE} and we omit it.
	
	\begin{lemma}
		There exists $\eta \in (0,1)$ such that
		\begin{equation}\label{decay_w}
			\omega_{1,\la}(r), \omega_{2,\la}(r) = O\Big(\sqrt{\la} e^{-\eta \sqrt{\la}|r - r_{\la}|}\Big), \ \text{ for any } r\geq 0.
		\end{equation}
	\end{lemma}
	
	
	Noting that if $u$ is radial, it holds that $\Delta u = u'' + \frac{1}{r} u'$. Then problem \eqref{eqmain} can be rewritten as
	\begin{equation}\label{eqmain_radial}
		\begin{cases}
			- u'' - \frac{1}{r}u' + (\la+P(r)) u= \alpha u^3 + \beta  uv^2,\  &\text{ in } \ \r,\\
			- v'' - \frac{1}{r}v'  +( \la+Q(r)) v= \gamma v^3 + \beta  u^2v,\  &\text{ in } \ \r,\\
            2\pi\int_{0}^{+\infty}r(u^2+v^2)dr = 1.  
		\end{cases} 
	\end{equation}
	In the following, we have the identity about $u_{\la}$ and $v_{\la}$.
	\begin{proposition}\label{identityprop1}
		If $(u_{\la},v_{\la})$ is a solution of \eqref{eqmain_radial} without the $L^2$-constraint, then we find
		\begin{equation}\label{identity1}
			\int_{0}^{+\infty}r\big( |u_\la'|^2+|v_\la'|^2 +(\la+P(r))u_\la^2 +(\la+Q(r))v_\la^2\big) dr 
			= \int_{0}^{+\infty} r(\alpha u_\la^4 +\gamma v_\la^4 +2\beta u_\la^2v_\la^2) dr.
		\end{equation}
	\end{proposition}
	
	\begin{proof}
		Multiplying $r u_{\la}$ and $r v_{\la}$ on both sides of $\eqref{eqmain_radial}_1$ and $\eqref{eqmain_radial}_2$ respectively and integrating from 0 to $+\infty$, we obtain \eqref{identity1}.
	\end{proof}
	
	Next, we can compute each term in \eqref{identity1}.
	
	\begin{proposition}\label{identityprop1jisuan}
		Let $(u_{\la},v_{\la})$ be a radial solution of \eqref{eqmain_radial} with the form \eqref{u_la}-\eqref{yx}, then we have the following estimates
		\begin{equation}\label{2-7}
			\begin{split}
				\int_{0}^{+\infty}r (|u_{\la}'|^{2}+|v_{\la}'|^{2}) d r =& \int_{0}^{+\infty}r (|U_{r_{\la},\la}'|^{2}+|V_{r_{\la},\la}'|^{2}) d r + O\Big( \la^{\frac{3}{4}}r_{\la}^{\frac{1}{2}} \|(\omega_{1,\la}, \omega_{2,\la})\|_\la+ \|(\omega_{1,\la}, \omega_{2,\la})\|_\la^2 \Big),
			\end{split}
		\end{equation}
	\begin{equation}\label{2-8}
		\begin{split}
				\int_{0}^{+\infty}r[(\la+P(r))u_{\la}^{2}+(\la+Q(r))v_{\la}^{2} ]dr 
			=& \int_{0}^{+\infty}r[(\la+P(r))U_{r_{\la},\la}^{2}+(\la+Q(r))V_{r_{\la},\la}^{2}] dr \\ 
			&+ O\Big( \la^{\frac{3}{4}}r_{\la}^{\frac{1}{2}} \|(\omega_{1,\la}, \omega_{2,\la})\|_\la+ \|(\omega_{1,\la}, \omega_{2,\la})\|_\la^2 \Big),
		\end{split}
	\end{equation}
and	\begin{equation}\label{2-9}
		\begin{split}
			\int_{0}^{+\infty}r (\alpha u_{\la}^{4} +\gamma v_{\la}^{4} + 2\beta u_{\la}^{2}v_{\la}^{2}) d r 
			=& \int_{0}^{+\infty}r (\alpha U_{r_{\la},\la}^{4} +\gamma V_{r_{\la},\la}^{4} + 2\beta U_{r_{\la},\la}^{2}V_{r_{\la},\la}^{2}) d r  \\ 
			&+ O\Big( \la^{\frac{3}{4}}r_{\la}^{\frac{1}{2}} \|(\omega_{1,\la}, \omega_{2,\la})\|_\la+ \|(\omega_{1,\la}, \omega_{2,\la})\|_\la^2 \Big).
		\end{split}
	\end{equation}

	\end{proposition}
	
	\begin{proof}
		By H\"older's inequality, we have
		\begin{equation*}
			\begin{split}
				&\int_{0}^{+\infty}r (|u_{\la}'|^{2}+|v_{\la}'|^{2}) d r \\
				= & \int_{0}^{+\infty}r (|U_{r_{\la},\la}'|^{2} +|V_{r_{\la},\la}'|^{2}) d r 
				+ 2 \int_{0}^{+\infty}r (U_{r_{\la},\la}'\omega_{1,\la}' + V_{r_{\la},\la}'\omega_{2,\la}') d r 
				 + \int_{0}^{+\infty}r ( |\omega_{1,\la}'|^{2} + |\omega_{2,\la}'|^{2}) d r \\
				=& \int_{0}^{+\infty}r (|U_{r_{\la},\la}'|^{2} +|V_{r_{\la},\la}'|^{2}) d r  + O\big(\|(U_{r_{\la},\la},V_{r_{\la},\la})\|_\la \|(\omega_{1,\la}, \omega_{2,\la})\|_\la 
				+ \|(\omega_{1,\la}, \omega_{2,\la})\|_\la^2 \big) \\
				=& \int_{0}^{+\infty}r (|U_{r_{\la},\la}'|^{2} +|V_{r_{\la},\la}'|^{2}) d r  
				+ O\big( \la^{\frac{3}{4}}r_\la^{\frac{1}{2}} \|(\omega_{1,\la}, \omega_{2,\la})\|_\la 
				+ \|(\omega_{1,\la}, \omega_{2,\la})\|_\la^2 \big),
			\end{split} 
		\end{equation*}     
which yields \eqref{2-7}. In a similar way, we obtain \eqref{2-8} by computing
	\begin{equation*}
	\begin{split}
		&\int_{0}^{+\infty}r[(\la+P(r))u_{\la}^{2}+(\la+Q(r))v_{\la}^{2} ]dr \\ 
		=& \int_{0}^{+\infty}r[(\la+P(r))U_{r_{\la},\la}^{2}+(\la+Q(r))V_{r_{\la},\la}^{2}] dr 
		+ O\Big( \la^{\frac{3}{4}}r_{\la}^{\frac{1}{2}} \|(\omega_{1,\la}, \omega_{2,\la})\|_\la+ \|(\omega_{1,\la}, \omega_{2,\la})\|_\la^2 \Big).
	\end{split}
	\end{equation*}
Using \eqref{decay_w}, it is straightforward to check that
	\begin{equation}\label{2-15}
	\begin{split}
		\int_{0}^{+\infty} ru_{\la}^{4}  dr
		=& \int_{0}^{+\infty}r \Big(U_{r_{\la},\la}^{4} 
		+ 4 U_{r_{\la},\la}^{3} \omega_{1,\la} 
		+ 6 U_{r_{\la},\la}^{2} \omega_{1,\la}^2 
		+ 4 U_{r_{\la},\la} \omega_{1,\la}^3 
		+ \omega_{1,\la}^4 \Big) dr  \\
		=& \int_{0}^{+\infty}r U_{r_{\la},\la}^{4} dr 
		+ O\Big( \la^{\frac{3}{4}}r_{\la}^{\frac{1}{2}} \|(\omega_{1,\la}, \omega_{2,\la})\|_\la +\|\omega_{1,\la} \|_{1,\la}^2\Big),
	\end{split}
	\end{equation}
where we have employed the fact
		$$\int_{0}^{+\infty} rU_{r_{\la},\la}^{3} \omega_{1,\la}d r 
		\leq C \la \int_{0}^{+\infty} rU_{r_{\la},\la} \omega_{1,\la}d r 
		\leq C \|U_{r_{\la},\la}\|_{1,\la} \| \omega_{1,\la} \|_{1,\la} 
		\leq C  \la^{\frac{3}{4}}r_{\la}^{\frac{1}{2}} \|\omega_{1,\la} \|_{1,\la},  $$
		$$\int_{0}^{+\infty} rU_{r_{\la},\la}^{2} \omega_{1,\la}^2d r 
		\leq C\int_{0}^{+\infty}r \la \omega_{1,\la}^2 dr 
		\leq C \|\omega_{1,\la} \|_{1,\la}^2, $$
		$$\int_{0}^{+\infty} rU_{r_{\la},\la} \omega_{1,\la}^3 d r 
		\leq C\int_{0}^{+\infty}r \la \omega_{1,\la}^2 dr 
		\leq C \|\omega_{1,\la} \|_{1,\la}^2, $$
		$$\int_{0}^{+\infty} r \omega_{1,\la}^4d r 
		\leq C\int_{0}^{+\infty}r \la \omega_{1,\la}^2 dr 
		\leq C \|\omega_{1,\la} \|_{1,\la}^2.$$
Similar to the computations of \eqref{2-15}, we can estimate $\int_{0}^{+\infty} rv_{\la}^{4}  dr$ and $\int_{0}^{+\infty} ru_{\la}^{2}v_\la^2  dr$, and finally derive \eqref{2-9}.
	\end{proof}
	
	Using above analysis, we can deduce a preliminary necessary condition.
	
	\begin{lemma}\label{radialprop}
		Let $(u_{\la},v_{\la})$ be a radial solution of \eqref{eqmain_radial} with the form \eqref{u_la}-\eqref{yx}, then we are able to obtain
		\begin{equation}\label{radial}
			\la^{\frac{1}{2}}r_\la \to +\infty, \quad \text{ as } \la \to +\infty.
		\end{equation}
		Moreover, if $(u_{\la},v_{\la})$ satisfies the $L^2$-constraint, then it must hold that $\frac{\alpha+\gamma-2\beta}{\alpha\gamma-\beta^2} \to 0$.
	\end{lemma}
	
	\begin{proof}
		From \eqref{u_la}-\eqref{yx}, Proposition \ref{identityprop1} and Proposition \ref{identityprop1jisuan}, we get
		\begin{equation*}
			\begin{split}
				{}&\int_{0}^{+\infty}r (|U_{r_{\la},\la}'|^{2}+|V_{r_{\la},\la}'|^{2}) d r 
				+\int_{0}^{+\infty}r[(\la+P(r))U_{r_{\la},\la}^{2}+(\la+Q(r))V_{r_{\la},\la}^{2}] dr \\
				={}& \int_{0}^{+\infty}r (\alpha U_{r_{\la},\la}^{4} +\gamma V_{r_{\la},\la}^{4} + 2\beta U_{r_{\la},\la}^{2}V_{r_{\la},\la}^{2}) d r
				+o(\la^{\frac{3}{2}}r_\la),
			\end{split}
		\end{equation*}
		which combining with \eqref{eqmain_approxi} gives us that
		\begin{equation}\label{2-16}
			\begin{split}
				\int_{0}^{+\infty}r (U_{r_{\la},\la} U_{r_{\la},\la}' + V_{r_{\la},\la}V_{r_{\la},\la}') d r  
				=o(\la^{\frac{3}{2}}r_\la),
			\end{split}
		\end{equation}
	where we used 
		\[
		\int_{0}^{+\infty}r (P(r)-P(r_\la)) U_{r_{\la},\la}^2 dr 
		\leq C \la^{-1}	\int_{0}^{+\infty}r \la U_{r_{\la},\la}^2 dr  
		\leq C \la^{-1}	 \|U_{r_{\la},\la}\|_{1,\la}^2  \leq C \la^{\frac{1}{2}}r_\la.
		\]
		Then using integration by parts in \eqref{2-16}, we deduce that
		\begin{equation}\label{2-17}
			\begin{split}
				U^{2}\Big( \sqrt{\la+P(r_\la)} r_{\la}\Big) + V^{2}\Big( \sqrt{\la+Q(r_\la)} r_{\la}\Big) 
				= o(\la^{\frac{1}{2}}r_\la).
			\end{split}
		\end{equation}
		If $\la^{\frac{1}{2}}r_\la \leq C_{0}$ for some fixed constant $C_0$, then \eqref{2-17} gives that $U^{2}(C_{0})+V^2(C_0) = 0$, which is impossible. Hence, we find \eqref{radial}. 
		
		Moreover, from the $L^2$-constraint, see $\eqref{eqmain_radial}_3$, one can check that
		\begin{equation}
			\begin{split}\label{L2}
				\frac{1}{2\pi} 
				={}& \int_{0}^{+\infty}r(u_\la^2+v_\la^2)dr \\
				={}& \int_{0}^{+\infty}r  (U_{r_{\la},\la}^2 +V_{r_{\la},\la}^2 )dr 
				+o(\la^{\frac{1}{2}}r_\la) \\
				={}& \la \int_{0}^{+\infty}r  \Big[U^{2}\Big( \sqrt{\la+P(r_\la)} (r-r_{\la})\Big) + V^{2}\Big( \sqrt{\la+Q(r_\la)} (r-r_{\la})\Big)\Big] dr 
				+o(\la^{\frac{1}{2}}r_\la) \\
				={}& \la r_\la \int_{0}^{+\infty}  \Big[U^{2}\Big( \sqrt{\la+P(r_\la)} (r-r_{\la})\Big) + V^{2}\Big( \sqrt{\la+Q(r_\la)} (r-r_{\la})\Big)\Big] dr 
				+o(\la^{\frac{1}{2}}r_\la) \\
				={}&  \frac{\alpha+\gamma-2\beta}{\alpha\gamma-\beta^2} \la^{\frac{1}{2}}r_\la \int_{-\infty}^{+\infty} w^2 dr 
				+o(\la^{\frac{1}{2}}r_\la) \\
			\end{split}
		\end{equation}
		which combining with \eqref{radial} gives that $\va=\frac{\alpha+\gamma-2\beta}{\alpha\gamma-\beta^2} \to 0$.
	\end{proof}

    \begin{remark}
     More precisely, \eqref{L2} can be improved to that
        \begin{equation}\label{L2improve}
            \frac{1}{2\pi\va} = \la^{\frac{1}{2}}r_\la \int_{-\infty}^{+\infty} w^2 dr 
				+o(\la^{\frac{1}{2}}r_\la). 
        \end{equation}
        In fact, the constructed solution $(u_\la,v_\la)$ can be refined as $u_\la= \sqrt{\frac{\gamma-\beta}{\alpha\gamma-\beta^2}} \bigg( \frac{1}{\sqrt{\frac{\gamma-\beta}{\alpha\gamma-\beta^2}}} U_{r_\la,\la} +\tilde{\omega}_{1,\la} \bigg)$ and $v_\la= \sqrt{\frac{\alpha-\beta}{\alpha\gamma-\beta^2}} \bigg( \frac{1}{\sqrt{\frac{\alpha-\beta}{\alpha\gamma-\beta^2}}} V_{r_\la,\la} +\tilde{\omega}_{2,\la} \bigg)$. Obviously, $\|(\tilde{\omega}_{1,\la}, \tilde{\omega}_{2,\la})\|_\la^2=\frac{\alpha\gamma-\beta^2}{\gamma-\beta}\|\omega_{1,\lambda}\|_{1,\lambda}^2+\frac{\alpha\gamma-\beta^2}{\alpha-\beta}\|\omega_{2,\lambda}\|_{2,\lambda}^2$.
    \end{remark}
	
	\begin{proposition}\label{lemmapoho}
		Let $(u_{\la},v_{\la})$ be a radial solution to \eqref{eqmain_radial} of the form \eqref{u_la}-\eqref{yx}, then the following conclusions hold
		\begin{equation}\label{2-10}
			\int_{0}^{+\infty}r(|u_{\la}'|^{2}+|v_{\la}'|^{2}) d r 
			= \frac{\alpha+\gamma-2\beta}{\alpha\gamma-\beta^2}A \la^{\frac{3}{2}}r_{\la} + o\big(\la^{\frac{3}{2}}r_{\la} \big),
		\end{equation}
	\begin{equation}\label{2-11}
		\int_{0}^{+\infty} r^2 [P'(r)u_\la^2+Q'(r)v_\la^2] dr
		= \frac{3A}{\alpha\gamma-\beta^2} [(\gamma-\beta)P'(r_\la)+(\alpha-\beta)Q'(r_\la)] \la^{\frac{1}{2}}r_{\la}^2 + o\big(\la^{\frac{1}{2}}r_{\la}^2 \big),
	\end{equation}
and\begin{equation}\label{2-12}
	\begin{split}
		\int_{0}^{+\infty} r (\alpha u_\la^4 + \gamma v_\la^4
		+2\beta u_\la^2 v_\la^2) dr 
		= 4\frac{\alpha+\gamma-2\beta}{\alpha\gamma-\beta^2}A \la^{\frac{3}{2}}r_{\la} + o\big(\la^{\frac{3}{2}}r_{\la} \big),
	\end{split}
\end{equation}
where $A:=\int_{-\infty}^{+\infty} |w'(r)|^2 dr.$
\end{proposition}
	
	\begin{proof}
		By a direct calculation, we deduce
		\begin{equation}\label{rU'2}
			\begin{split}
				 \int_{0}^{+\infty}r|U_{r_{\la},\la}'|^{2}d r 
				={}& (\la+P(r_\la))^2  \int_{0}^{+\infty}r  \Big| U'\Big( \sqrt{\la+P(r_\la)} (r-r_{\la})\Big) \Big|^2 dr \\
				={}& (\la+P(r_\la))^{\frac{3}{2}} r_{\la}  \int_{0}^{+\infty}  | U' |^2 dr + O(\la) \\
				={}& \frac{\gamma-\beta}{\alpha\gamma-\beta^2} \la^{\frac{3}{2}} r_{\la} \Big(1+\frac{P(r_\la)}{\la}\Big)^{\frac{3}{2}}   \int_{-\infty}^{+\infty} |w'(r)|^2 dr + O(\la)  \\
				={}& \frac{\gamma-\beta}{\alpha\gamma-\beta^2} A \la^{\frac{3}{2}} r_{\la} + O(\la) .
\end{split}
\end{equation}
	From a similar argument, it yields
	\begin{equation}
    \label{rV'2}
		\begin{split}
			\int_{0}^{+\infty}r|V_{r_{\la},\la}'|^{2}d r 
			=& \frac{\alpha-\beta}{\alpha\gamma-\beta^2} A \la^{\frac{3}{2}} r_{\la} + O(\la). 
		\end{split}
	\end{equation}
Using \eqref{2-7}, \eqref{rU'2} and \eqref{rV'2}, We obtain \eqref{2-10} via the following step 
		\begin{equation*}
			\begin{split}
			 \int_{0}^{+\infty}r(|u_{\la}'|^{2}+|v_{\la}'|^{2}) d r 
			 &= \int_{0}^{+\infty}r(|U_{r_{\la},\la}'|^{2}+|V_{r_{\la},\la}'|^{2}) d r 
			 + o\big(\la^{\frac{3}{2}}r_{\la} \big) \\
			&= \frac{\alpha+\gamma-2\beta}{\alpha\gamma-\beta^2}A \la^{\frac{3}{2}}r_{\la} + o\big(\la^{\frac{3}{2}}r_{\la} \big).
		\end{split}
	\end{equation*}

		Combining with Lemma \ref{pohozaev}, we further obtain
		\begin{equation}\label{2-18}
			\begin{split}
				\int_{0}^{+\infty}r^{2}P'(r)U_{r_{\la},\la}^{2}d r 
				={}& (\la+P(r_\la))  \int_{0}^{+\infty}r^{2}P'(r_{\la})U^{2}\Big( \sqrt{\la+P(r_\la)} (r-r_{\la})\Big) d r  \\
				&+ O\Big( \la \int_{0}^{+\infty}r^{2}\cdot |r - r_{\la}|\cdot U^{2}\Big( \sqrt{\la+P(r_\la)} (r-r_{\la})\Big) d r\Big)\\
				={}& P'(r_\la)r_\la^2(\la+P(r_\la))^{\frac{1}{2}} \int_{-\infty}^{+\infty} U^2 dr +O(r_\la^2)   \\
				={}& 3\frac{\gamma-\beta}{\alpha\gamma-\beta^2} AP'(r_\la) \la^{\frac{1}{2}} r_\la^2 + O(r_\la^2),
			\end{split}
		\end{equation}
		\begin{equation}\label{2-19}
			\begin{split}
				\int_{0}^{+\infty}r^{2}P'(r)U_{r_{\la},\la} \omega_{1,\la} d r 
				\leq &  C \frac{1}{\la} 	\Big(\int_{0}^{+\infty}r^{3} \la |P'(r)|^2 U_{r_{\la},\la}^2  d r \Big)^{\frac{1}{2}} \|\omega_{1,\la}\|_{1,\la} 
				\leq C\la^{-\frac{1}{4}} r_\la^{\frac{3}{2}} \|\omega_{1,\la}\|_{1,\la}  , 
			\end{split}
		\end{equation}
	and 
		\begin{equation}\label{2-20}
			\begin{split}
				\int_{0}^{+\infty}r^{2}P'(r) \omega_{1,\la}^2 d r 
				=&  \frac{1}{\la} \int_{0}^{+\infty}r^{2}  P'(r) \la   \omega_{1,\la}^2  d r \\
				=&  \frac{r_\la}{\la} \int_{0}^{+\infty}r  P'(r) \la   \omega_{1,\la}^2  d r 
				+ \frac{1}{\la} \int_{0}^{+\infty}r(r-r_\la)  P'(r) \la   \omega_{1,\la}^2  d r \\
				=& O\Big( \frac{r_\la}{\la} \|\omega_{1,\la}\|_{1,\la} ^2\Big).
	\end{split}
 \end{equation}
		Along the same lines, we have
		\begin{equation}\label{2-21}
			\begin{split}
				\int_{0}^{+\infty}r^{2}Q'(r)V_{r_{\la},\la}^{2}d r 
				=& 3\frac{\alpha-\beta}{\alpha\gamma-\beta^2} AQ'(r_\la) \la^{\frac{1}{2}} r_\la^2 + O(r_\la^2),
			\end{split}
		\end{equation}
		\begin{equation}\label{2-22}
			\begin{split}
				\int_{0}^{+\infty}r^{2}Q'(r)V_{r_{\la},\la} \omega_{2,\la} d r 
				\leq C\la^{-\frac{1}{4}} r_\la^{\frac{3}{2}} \|\omega_{2,\la}\|_{2,\la}  , 
			\end{split}
		\end{equation}
		and 		\begin{equation}\label{2-23}
			\begin{split}
				\int_{0}^{+\infty}r^{2}Q'(r) \omega_{2,\la}^2 d r 
				=& O\Big( \frac{r_\la}{\la} \|\omega_{2,\la}\|_{2,\la} ^2\Big).
	\end{split}
\end{equation}
	Thus \eqref{2-11} follows from \eqref{2-8} and \eqref{2-18}-\eqref{2-23}.
		
		Additionally, by \eqref{2-9} and Lemma \ref{pohozaev}, we calculate that
		\begin{equation}\label{2-24}
				\int_{0}^{+\infty}ru_\la^{4}d r 
				= \int_{0}^{+\infty}r U_{r_{\la},\la}^4 dr + o(\la^{\frac{3}{2}} r_\la ) 
				=4 \Big(\frac{\gamma-\beta}{\alpha\gamma-\beta^2}\Big)^2 A \la^{\frac{3}{2}} r_\la  + o(\la^{\frac{3}{2}} r_\la ),
		\end{equation}
		\begin{equation}\label{2-25}
			\begin{split}
				\int_{0}^{+\infty}rv_\la^{4}d r 
				=&4 \Big(\frac{\alpha-\beta}{\alpha\gamma-\beta^2}\Big)^2 A \la^{\frac{3}{2}} r_\la  + o(\la^{\frac{3}{2}} r_\la ),
			\end{split}
		\end{equation}
and		\begin{equation}\label{2-26}
			\begin{split}
				\int_{0}^{+\infty}ru_\la^{2}v_\la^{2}d r 
				=&4 \frac{(\alpha-\beta)(\gamma-\beta)}{(\alpha\gamma-\beta^2)^2} A \la^{\frac{3}{2}} r_\la  + o(\la^{\frac{3}{2}} r_\la ).
			\end{split}
		\end{equation}
		Consequently, \eqref{2-12} follows from \eqref{2-24}-\eqref{2-26}.
	\end{proof}
	
	Furthermore, we can establish the following Pohozaev identity on $(u_{\la},v_{\la})$.
	
	\begin{proposition}\label{poho}
		If $(u_{\la},v_{\la})$ is a solution of \eqref{eqmain_radial}, then we have
		\begin{equation}\label{poho2}
			\begin{split}
				2 \int_{0}^{+\infty} r [|u_\la'|^2+|v_\la'|^2 ]dr 
				- \int_{0}^{+\infty} r^2 [P'(r)u_\la^2+Q'(r)v_\la^2] dr
				=&  \int_{0}^{+\infty} r (\alpha u_\la^4 + \gamma v_\la^4
				+2\beta  r u_\la^2 v_\la^2) dr .
			\end{split}
		\end{equation}
	\end{proposition}
	
	\begin{proof}
		Multiplying $r^{2}u_{\la}'$, $r^{2}v_{\la}'$ on both sides of $\eqref{eqmain_radial}_1$, $\eqref{eqmain_radial}_2$ respectively,  and then integrating over $[0,+\infty)$, we arrive at
		\begin{equation}\label{2-27}
			\int_{0}^{+\infty} r^2 \Big[-u_\la''-\frac{1}{r}u_\la'+(\la+P(r))u_\la -\alpha u_\la^3 -\beta u_\la v_\la^2 \Big] u_\la' dr =0,
		\end{equation}
	and\begin{equation}\label{2-271}
		\int_{0}^{+\infty} r^2 \Big[-v_\la''-\frac{1}{r}v_\la'+(\la+Q(r))v_\la -\gamma v_\la^3 -\beta u_\la^2 v_\la \Big] v_\la' dr =0.
	\end{equation}
		By integration by parts and the decay of $u_\la$, we have
		\begin{equation*}
			-\int_{0}^{+\infty} r^2 u_\la'' u_\la' dr  
			= \int_{0}^{+\infty} r |u_\la'|^2 dr,
		\end{equation*}
		\begin{equation*}
			\int_{0}^{+\infty} r^2 (\la+P(r))u_\la u_\la' dr  
			= - \int_{0}^{+\infty} r (\la+P(r))u_\la^2  dr
			 - \frac{1}{2} \int_{0}^{+\infty} r^2 P'(r) u_\la^2  dr   ,
		\end{equation*}
		\begin{equation*}
			\int_{0}^{+\infty} r^2 u_\la^3 u_\la' dr  
			= - \frac{1}{2} \int_{0}^{+\infty} r  u_\la^4  dr   ,
		\end{equation*}
		and
		\begin{equation*}
			\int_{0}^{+\infty} r^2 u_\la v_\la^2 u_\la' dr  
			=  \frac{1}{2} \int_{0}^{+\infty} r^2  (u_\la^2)' v_\la^2  dr  .
		\end{equation*}
		Therefore, \eqref{2-27} becomes
		\begin{equation}\label{2-28}
			 2 \int_{0}^{+\infty} r (\la+P(r))u_\la^2  dr 
			 = \alpha \int_{0}^{+\infty} r  u_\la^4 dr
			 	- \int_{0}^{+\infty} r^2 P'(r) u_\la^2  dr 
			 	- \beta  \int_{0}^{+\infty} r^2  (u_\la^2)' v_\la^2  dr.
		\end{equation}		
		A parallel argument gives
		\begin{equation}\label{2-29}
			2 \int_{0}^{+\infty} r (\la+Q(r))v_\la^2  dr 
			= \gamma \int_{0}^{+\infty} r  v_\la^4 dr
			- \int_{0}^{+\infty} r^2 Q'(r) v_\la^2  dr 
			- \beta  \int_{0}^{+\infty} r^2  u_\la^2 (v_\la^2)'  dr.
		\end{equation}
		Adding equations \eqref{2-28} and \eqref{2-29} yields 
		\begin{equation}\label{2-289}
			\begin{split}
				2 \int_{0}^{+\infty} r \big[(\la+P(r))u_\la^2  + (\la+Q(r))v_\la^2 \big] dr 
				=&  \int_{0}^{+\infty} r  (\alpha u_\la^4 + \gamma v_\la^4) dr
				+ 2\beta  \int_{0}^{+\infty} r  u_\la^2 v_\la^2  dr \\
				&- \int_{0}^{+\infty} r^2 [P'(r) u_\la^2 + Q'(r) v_\la^2]  dr.
			\end{split}
		\end{equation}
		Then combining \eqref{identity1} with \eqref{2-289}, we conclude \eqref{poho2}.
	\end{proof}
	
	We are now ready to establish Theorem \ref{diyigedingli}.
	
	\begin{proof}[\textbf{Proof of Theorem \ref{diyigedingli}}]
		First, from Lemma \ref{radialprop}, we find $\frac{\alpha+\gamma-2\beta}{\alpha\gamma-\beta^2} \to 0$ and
		\[
		\la^{\frac{1}{2}}r_\la \to +\infty, \,\,\text{ as } \la \to +\infty.
		\]
		Next, from Proposition \ref{lemmapoho} and \eqref{poho2}, we get
		\begin{equation*}
			2A \frac{\alpha+\gamma-2\beta}{\alpha\gamma-\beta^2} \la^{\frac{3}{2}}r_{\la} +  \frac{3A}{\alpha\gamma-\beta^2} [(\gamma-\beta)P'(r_\la)+(\alpha-\beta)Q'(r_\la)] \la^{\frac{1}{2}}r_{\la}^2 
			= o\big(\la^{\frac{1}{2}}r_{\la}^2 \big) + o\big(\la^{\frac{3}{2}}r_{\la} \big),
		\end{equation*}
	namely,
	\begin{equation}\label{2-13}
		\alpha+\gamma-2\beta +  \frac{3}{2} \Big[(\gamma-\beta) P'(r_\la)+ (\alpha-\beta) Q'(r_\la) \Big] \frac{r_\la}{\la} 
		= o\big(\frac{r_\la}{\la} \big) + o(1).
	\end{equation}
		If we suppose that $r_{\la}$ is bounded, then \eqref{2-13} gives us that
		\[
		1 = o(1),
		\]
		which is impossible. Hence it follows that $\lim\limits_{\la \to +\infty}r_{\la} = +\infty$.
		
		Now, if $\frac{r_\la}{\la} $ is finite, then \eqref{2-13} is equivalent to
\begin{equation*}\label{2-14}
			1 +  \frac{3}{2} \Big(\frac{\gamma-\beta}{\alpha+\gamma-2\beta}  P'(r_\la)+ \frac{\alpha-\beta}{\alpha+\gamma-2\beta} Q'(r_\la) \Big) \frac{r_\la}{\la} 
			= o(1).
		\end{equation*}
	Recalling the definition of $M_\la(r)$ given in \eqref{Mdingyi} 
	\[
	M_\la(r):=r\Big[1+\frac{1}{\la}\Big(\frac{\gamma-\beta}{\alpha+\gamma-2\beta} P(r) + \frac{\alpha-\beta}{\alpha+\gamma-2\beta} Q(r) \Big)\Big]^{\frac{3}{2}},
	\]
	we have
	\begin{equation*}
		\begin{split}
			 M_{\la}'(r_{\la})  
			={}& 
			\bigg(1+\frac{1}{\la}\Big(\frac{\gamma-\beta}{\alpha+\gamma-2\beta} P(r_\la) + \frac{\alpha-\beta}{\alpha+\gamma-2\beta} Q(r_\la) \Big)\bigg)^{\frac{1}{2}} \\
			&\cdot\bigg(
			1 +  \frac{3}{2} \frac{r_\la}{\la} \Big(\frac{\gamma-\beta}{\alpha+\gamma-2\beta}  P'(r_\la)+ \frac{\alpha-\beta}{\alpha+\gamma-2\beta} Q'(r_\la) \Big)  \\
			&\ \ \ \ +\frac{1}{\la}\Big(\frac{\gamma-\beta}{\alpha+\gamma-2\beta} P(r_\la) + \frac{\alpha-\beta}{\alpha+\gamma-2\beta} Q(r_\la) \Big)
			\bigg)
			\to 0, \quad \text{as } \la \to +\infty.
		\end{split}
	\end{equation*}
		
		On the other hand, if $\frac{r_\la}{\la} \to +\infty$ as $\la \to +\infty$, then  \eqref{2-13} is equivalent to
		\begin{equation*}
			(\gamma-\beta) P'(r_\la)+ (\alpha-\beta) Q'(r_\la) =o(1).
		\end{equation*}
		This completes the proof of Theorem \ref{diyigedingli}.
	\end{proof}
	
	\section{Existence}\label{cz1}
	In this section, we aim to prove the existence of blowing-up solutions of problem \eqref{eqmain}, which has the form
	\begin{equation*}
		(u_\la(r),v_\la(r))=(U_{r_{\la},\la}(r),V_{r_{\la},\la}(r))+(\omega_{1,\la}(r), \omega_{2,\la}(r)),
	\end{equation*}
	where $$U_{r_{\la},\la}(r)= \sqrt{\la +P(r_\la)} U(\sqrt{\lambda+P(r_\la)} (r - r_{\la})),$$
	$$V_{r_{\la},\la}(r)= \sqrt{\la +Q(r_\la)} V(\sqrt{\lambda+Q(r_\la)} (r - r_{\la}))$$ and $\|(\omega_{1,\la}(r), \omega_{2,\la}(r))\|_\la=o(\la^{\frac{3}{4}}r_{\la}^{\frac{1}{2}})$.

 
    We introduce the Hilbert space $H_r^1(\r^2)$
    	\[
    		H_{r}^{1}(\mathbb{R}^{2}) := \Big\{u\in H^{1}(\mathbb{R}^{2}):u(x) = u(y) \text{ if } |x| = |y|\Big\}.
    	\]
	For any $r_{\la}\in \mathbb{R}^{+}$, the space $E_{r_{\la},\la}$ is defined by
	\begin{equation*}
		E_{r_{\la},\la} = \left\{  (u,v)  \in H_{r}^{1}(\mathbb{R}^{2})\times H_{r}^{1}(\mathbb{R}^{2}) : 
	    \left\langle (u, v), (U_{r_{\la},\la}',V_{r_{\la},\la}') \right\rangle_{\la}= 0\right\}.
	\end{equation*}
	Now, for any $r_\la$ large enough,  we consider the following problem
	\begin{equation}
    \label{uv}
		L_{\lambda}(\omega_{1},\omega_{2}) = l_{\lambda} + R_{\lambda}(\omega_{1},\omega_{2}), \ \text{ for any } (\omega_{1},\omega_{2})  \in H_{r}^{1}(\mathbb{R}^{2})\times H_{r}^{1}(\mathbb{R}^{2}),
	\end{equation}
	where $L_{\lambda}$ is the bounded linear operator from $H_{r}^{1}(\mathbb{R}^{2}) \times H_{r}^{1}(\mathbb{R}^{2})$ to itself, defined by
	\begin{equation}\label{L}
		\begin{split}
			&\langle L_{\lambda}(\omega_{1},\omega_{2}),(\phi_{1},\phi_{2})\rangle_{\lambda} \\
			&=\int_{0}^{+\infty}r
			\Big[(\sum_{i=1}^{2}\omega_{i}'\phi_{i}') + (\lambda + P(r))\omega_{1}\phi_{1} +(\lambda + Q(r))\omega_{2}\phi_{2}  
			- 3\alpha U_{r_{\la},\la}^{2}\omega_{1}\phi_{1}
			- 3\gamma V_{r_{\la},\la}^{2}\omega_{2}\phi_{2}\Big] dr \\
			&\quad -\int_{0}^{+\infty}r 
			\Big( \beta V_{r_{\la},\la}^{2}\omega_{1}\phi_{1} + 2\beta U_{r_{\la},\la}V_{r_{\la},\la} \omega_{2}\phi_{1} \Big) dr 
			 -\int_{0}^{+\infty}r 
			\Big( \beta U_{r_{\la},\la}^{2}\omega_{2}\phi_{2} + 2\beta U_{r_{\la},\la}V_{r_{\la},\la} \omega_{1}\phi_{2} \Big) dr , 
		\end{split}
	\end{equation}
	$l_{\lambda}:=(l_{1,\lambda},l_{2,\lambda})$ is defined by
	\begin{align*}
		\langle l_\la,(\phi_{1},\phi_{2})\rangle_{\lambda} 
		= \int_{0}^{+\infty}r \Big[(P(r_\la)-P(r)) U_{r_{\la},\la}\phi_1 + (Q(r_\la)-Q(r)) V_{r_{\la},\la} \phi_2+ \frac{1}{r}U_{r_{\la},\la}' \phi_1+\frac{1}{r} V_{r_{\la},\la}' \phi_2 \Big] dr,
	\end{align*}
	and $R_{\lambda}(\omega_{1},\omega_{2}):=(R_{1,\lambda}(\omega_{1},\omega_{2}),R_{2,\lambda}(\omega_{1},\omega_{2}))$ is defined by
	\begin{align*}
		R_{1,\lambda}(\omega_{1},\omega_{2}) &= \alpha \Big(\omega_{1}^{3} + 3 U_{r_{\la},\la} \omega_{1}^{2}\Big) + \beta\Big(\omega_{1}\omega_{2}^{2} +  U_{r_{\la},\la} \omega_{2}^{2} + 2 V_{r_{\la},\la} \omega_{1}\omega_{2}\Big),\\
		R_{2,\lambda}(\omega_{1},\omega_{2}) &= \gamma \Big(\omega_{2}^{3} + 3 V_{r_{\la},\la} \omega_{2}^{2}\Big) + \beta\Big(\omega_{2}\omega_{1}^{2} +  V_{r_{\la},\la} \omega_{1}^{2} + 2 U_{r_{\la},\la} \omega_{1}\omega_{2}\Big).
	\end{align*}

	We have the following invertibility of the operator $L_{\la}$ on $E_{r_{\la},\la}$.
	
	\begin{proposition}\label{invertibility}There exist constants $\rho >0$ and $\la_0>0$, such that for any $\la >\la_0$, it holds that
		\begin{equation*}
			\|L_{\lambda}(\omega_{1},\omega_{2}) \|_\la \geq \rho \|(\omega_{1},\omega_{2})\|_\la,  \ \text{ for any } (\omega_{1},\omega_{2})  \in E_{r_{\la},\la}.
		\end{equation*}
	\end{proposition}
	
	\begin{proof}
		We use a contradiction argument to prove it. Suppose that there exist $\la_{n}\to +\infty$, $r_{\la_n}\in \mathbb{R}^+$ and $\omega_n=(\omega_{1,n},\omega_{2,n})  \in E_{r_{n},\la_n}$ with $\|\omega_n\|_{\la_n}=1$ such that
		\begin{equation}\label{3-1}
			\|L_{\lambda_n}(\omega_{1,n},\omega_{2,n}) \|_{\la_n}
			= o_n(1).
		\end{equation}
		It is a bit standard to obtain a contradiction from \eqref{3-1}. So we just sketch the proof. For simplicity, we drop the subscript $n$. Then we have, for any $(\phi_1,\phi_2)\in E_{r_\lambda,\lambda}$,
		\begin{equation}\label{3-2}
			\begin{split}
				& o(1)\|(\phi_{1},\phi_{2})\|_\la \\
				& =\int_{0}^{+\infty}r
				\Big[(\sum_{i=1}^{2}\omega_{i}'\phi_{i}') + (\lambda + P(r))\omega_{1}\phi_{1} +(\lambda + Q(r))\omega_{2}\phi_{2}  
				- 3\alpha U_{r_{\la},\la}^{2}\omega_{1}\phi_{1}
				- 3\gamma V_{r_{\la},\la}^{2}\omega_{2}\phi_{2}\Big] dr \\
				&\quad -\int_{0}^{+\infty}r 
				\Big( \beta V_{r_{\la},\la}^{2}\omega_{1}\phi_{1} + 2\beta U_{r_{\la},\la}V_{r_{\la},\la} \omega_{2}\phi_{1} \Big) dr  
				-\int_{0}^{+\infty}r 
				\Big( \beta U_{r_{\la},\la}^{2}\omega_{2}\phi_{2} + 2\beta U_{r_{\la},\la}V_{r_{\la},\la} \omega_{1}\phi_{2} \Big) dr .
			\end{split}
		\end{equation}
		Let $(\phi_{1},\phi_{2}) = (\omega_{1},\omega_{2}) $ in \eqref{3-2}, we get
		\begin{equation}\label{3-3}
			\begin{split}
				 o(1)
				& =\int_{0}^{+\infty}r
				\Big[(\sum_{i=1}^{2}|\omega_{i}'|^2) + (\lambda + P(r))\omega_{1}^{2} +(\lambda + Q(r))\omega_{2}^{2}  
				- 3\alpha U_{r_{\la},\la}^{2}\omega_{1}^{2}
				- 3\gamma V_{r_{\la},\la}^{2}\omega_{2}^{2}\Big] dr \\
				&\quad -\int_{0}^{+\infty}r 
				\Big( \beta V_{r_{\la},\la}^{2}\omega_{1}^{2} + 2\beta U_{r_{\la},\la}V_{r_{\la},\la} \omega_{2}\omega_{1} \Big) dr 
				 -\int_{0}^{+\infty}r 
				\Big( \beta U_{r_{\la},\la}^{2}\omega_{2}^{2} + 2\beta U_{r_{\la},\la}V_{r_{\la},\la} \omega_{1}\omega_{2} \Big) dr,
			\end{split}
		\end{equation}
	namely,
		\begin{equation}\label{3-4}
			\begin{split}
				o(1)
				 ={}& 1- 3 \int_{0}^{+\infty}r
				\Big[
				 \alpha U_{r_{\la},\la}^{2}\omega_{1}^{2}
				+ \gamma V_{r_{\la},\la}^{2}\omega_{2}^{2}\Big] dr  
				-\int_{0}^{+\infty}r 
				\Big( \beta V_{r_{\la},\la}^{2}\omega_{1}^{2} + 2\beta U_{r_{\la},\la}V_{r_{\la},\la} \omega_{2}\omega_{1} \Big) dr \\
				& -\int_{0}^{+\infty}r 
				\Big( \beta U_{r_{\la},\la}^{2}\omega_{2}^{2} + 2\beta U_{r_{\la},\la}V_{r_{\la},\la} \omega_{1}\omega_{2} \Big) dr,
			\end{split}
		\end{equation}
		By the exponential decay of $U,V$, for fixed $R > 0$ large enough, we have
		\begin{equation*}
			U_{r_{\la},\la}^{a}, V_{r_{\la},\la}^{a}  = o_R(1) \la^{\frac{a}{2}}, \quad \text{in } (0, + \infty)\backslash [r_\la - R/\sqrt{\la},r_\la + R/\sqrt{\la}],
		\end{equation*}
	where $a>0$ is a constant. So we will obtain the contradiction, if we can prove
	\begin{equation}\label{3-5}
		\lambda \int_{r_\la - R/\sqrt{\la}}^{r_\la + R/\sqrt{\la}} r (\omega_{1}^{2}+\omega_{2}^{2})dr  = o(1) .
	\end{equation}
	
	To prove \eqref{3-5}, we define
		$\omega_{\la,i}(r) = \omega_{i}\Big(\frac{r}{\sqrt{\la}}+r_\la\Big).$
	Then by $\|\omega\|_\la=1$, we have
	\begin{equation*}
		\int_{0}^{+\infty} r (|\omega_{\lambda,i}'|^2 +|\omega_{\lambda,i}|^2)  dr \leq C, \quad i=1,2.
	\end{equation*}
	Thus, there exists $\bar{\omega}_{1}\in H_r^1(\r^2)$ and $\bar{\omega}_{2}\in H_r^1(\r^2)$ such that, as $\la\to +\infty$,
	\begin{equation*}
		\omega_{\la,i} \rightharpoonup \bar{\omega}_{i} \text{ weakly in } H_r^1(\r^2),
	\end{equation*}
	and
	\begin{equation*}
		\omega_{\la,i} \rightarrow \bar{\omega}_{i} \text{ strongly in } L^2_{r,loc}(\r^2).
	\end{equation*}
	It is easy to see that  \eqref{3-5} holds if $\bar{\omega}_{i}=0, i=1,2$.
	By \eqref{3-2}, it is standard to prove that $(\bar{\omega}_1,\bar{\omega}_2)$ satisfies the following system
	\begin{equation*}\label{3-6}
		\begin{cases}
			- \bar{\omega}_1'' +  \bar{\omega}_1 = 3\alpha U^2\bar{\omega}_1 + \beta V^2 \bar{\omega}_1  +2 \beta UV\bar{\omega}_2,&\text{in}\,\,\r,\\
			- \bar{\omega}_2'' +  \bar{\omega}_2= 3\gamma V^2\bar{\omega}_1 + \beta U^2 \bar{\omega}_1  +2 \beta UV\bar{\omega}_1,  &\text{in}\,\,\r.\\
		\end{cases}
	\end{equation*}
	Hence, the non-degeneracy of $(U,V)$ gives that
\[
(\bar{\omega}_1, \bar{\omega}_2 )= d \Big(\theta(\beta)\frac{\partial w}{\partial r}, \frac{\partial w}{\partial r}  \Big) .
\]

On the other hand, from  $(\omega_{1},\omega_{2})  \in E_{r_\la,\la}$,  we obtain
\begin{equation*}\label{2-1-24}\begin{split}
		0 = {}&  \big\langle (\omega_{1},\omega_{2}), (U_{r_{\la},\la}',V_{r_{\la},\la}') \big\rangle_{\la}  \\
		={}& \int_{0}^{+\infty} r \bigl( \omega_{1}' U_{r_{\la},\la}'' +  \omega_{2}' V_{r_{\la},\la}''
		 +(\la+P(r)) \omega_{1} U_{r_{\la},\la}'  +(\la+Q(x)) \omega_{2} V_{r_{\la},\la}'  \bigr) dr    \\
		={}& \int_{0}^{+\infty}r  \Bigg\{
		\Big[3\alpha U_{r_{\la},\la}^{2}U_{r_{\la},\la}' 
		+\beta U_{r_{\la},\la}' V_{r_{\la},\la}^{2} 
		+ 2\beta U_{r_{\la},\la} V_{r_{\la},\la} V_{r_{\la},\la}' \Big] \omega_{1} 
		+ \Big[3\gamma V_{r_{\la},\la}^2 V_{r_{\la},\la}' 
		+\beta V_{r_{\la},\la}'  U_{r_{\la},\la}^2 \\
		&\quad\qquad+ 2\beta U_{r_{\la},\la} V_{r_{\la},\la} U_{r_{\la},\la}'  \Big] \omega_{2}  
		+ (P(r)-P(r_\la)) \omega_{1} U_{r_{\la},\la}'
		+ (Q(r)-Q(r_\la)) \omega_{2} V_{r_{\la},\la}'  \Bigg\} dr \\
		={}& \lambda  \biggl(  \int_{0}^{+\infty}r  
		\Big[
        \Big(3\alpha U^{2}  U' +\beta U' V^2 +2\beta UVV' \Big) \omega_{\la,1}
		+\Big(3\gamma V^{2}  V' +\beta V' U^2 +2\beta UVU' \Big) \omega_{\la,2}
        \Big] dr  +o(1)\biggr).
	\end{split}
\end{equation*}
By the definition of $(U,V)$, we have 
\begin{equation*}\label{3-7}
	\int_{0}^{+\infty}r  w^{2}w' (\omega_{\la,1}  + \omega_{\la,2}  ) dr= o(1),
\end{equation*}
which implies that
\begin{equation*}\label{3-8}
	\int_{0}^{+\infty}r  w^{2}w' (\bar{\omega}_{1}  + \bar{\omega}_{2}  ) dr= 0.
\end{equation*}
Hence, $d = 0$. This shows that $\bar{\omega}_{i}=0$, $i=1,2$. 
The proof is completed.
	\end{proof}
	
%
	
	From now on, we always assume that $r_\la\in [\la^{1-\theta}, \la^{1+\theta}]$, where $\theta>0$ is a fixed small constant. 
	We give the following lemmas which give the estimates of $l_{\la}$ and $R_{\la}(\omega_{1},\omega_{2})$.
	
	\begin{lemma}\label{l_aaa}  There is a constant $C>0$, independent of $\la$, such that
		\begin{equation}\label{l}
			\|l_{\la}\|_{\la} =O \Bigl(\la^{-\frac{1}{4}+\frac{\theta}{2}}\Bigr) .
		\end{equation}
	\end{lemma}
	
	\begin{proof}
		For any $(\varphi,\psi) \in 
        H^1(\r^2)\times H^1(\r^2)$, we have
		\[
		\langle (l_{1,\la}, l_{2,\la}), (\varphi,\psi)\rangle_{\la} = \langle l_{1,\la}, \varphi\rangle_{1,\la} +  \langle l_{2,\la}, \psi \rangle_{2,\la}.
		\]
		By direct calculation, we have
		\begin{equation*}
		\begin{split}
			\langle l_{1,\la}, \varphi\rangle_{1,\la} ={}& \int_{0}^{+\infty} r \Big[(P(r_\la)-P(r))U_{r_\la, \la}  
			+ \frac{1}{r} U_{r_\la, \la}' \Big]  \varphi dr \\
			\leq{}&  \frac{C}{\la} \bigg[ \Big(\int_{0}^{+\infty} r \la |P(r_\la)-P(r)|^2 |U_{r_\la, \la}|^2  dr \Big)^{\frac{1}{2}}
			+  \Big(\int_{0}^{+\infty} \frac{1}{r} \la |U_{r_\la, \la}'|^2 dr \Big)^{\frac{1}{2}}  \bigg] 
			\cdot \|\varphi\|_{1,\la}  \\
			\leq{}&  C \la^{-\frac{1}{2}} \bigg[ \Big(\int_{0}^{+\infty} r  |r_\la-r|^2 |U_{r_\la, \la}|^2  dr \Big)^{\frac{1}{2}}
			+  \Big(\int_{0}^{+\infty} \frac{1}{r}  |U_{r_\la, \la}'|^2 dr \Big)^{\frac{1}{2}}  \bigg] 
			\cdot \|\varphi\|_{1,\la}.
		\end{split}
	\end{equation*}
		Since 
		\begin{equation*}
			\begin{split}
				\int_{0}^{+\infty} r  |r_\la-r|^2 |U_{r_\la, \la}|^2  dr  
				={}& (\la+P(r_\la))^{-\frac{1}{2}}\int_{-\sqrt{\la+P(r_\la)}r_\la}^{+\infty} (\frac{r}{\sqrt{\la+P(r_\la)}}+r_\la)  |r|^2 |U(r)|^2  dr \\
				\leq{}&C \la^{-\frac{1}{2}} r_\la \int_{-\infty}^{+\infty} r^2U^2(r) dr + O(\la^{-1}) \\
				\leq{}& C  \la^{-\frac{1}{2}} r_\la,
			\end{split}
		\end{equation*}
		and 
		\begin{equation*}
			\begin{split}
				\int_{0}^{+\infty} \frac{1}{r}  |U_{r_\la, \la}'|^2 dr 
				={}& \int_{r_\la-R}^{r_\la+R} \frac{1}{r}  |U_{r_\la, \la}'|^2 dr + O(e^{-\frac{1}{4}\sqrt{\la}R}) \\
				={}& \int_{r_\la-R}^{r_\la+R} \frac{1}{r_\la}  |U_{r_\la, \la}'|^2 dr + O(\la^{\frac{3}{2}}r_\la^{-2}) \\
				\leq{}& C \la^{\frac{3}{2}}r_\la^{-1} ,
			\end{split}
		\end{equation*}
	we therefore have 
	\begin{equation*}
		\begin{split}
			\langle l_{1,\la}, \varphi\rangle_{1,\la} 
			\leq  C \la^{-\frac{1}{2}} \bigg[ \la^{-\frac{1}{4}} r_\la^{\frac{1}{2}}
			+ \la^{\frac{3}{4}}r_\la^{-\frac{1}{2}}  \bigg] 
			\cdot \|\varphi\|_{1,\la} 
			\leq  C  \la^{-\frac{1}{4}+\frac{\theta}{2}}  \|\varphi\|_{1,\la}.  \\
		\end{split}
	\end{equation*}
	Similarly, we can obtain that 
		\begin{equation*}
		\begin{split}
			\langle l_{2,\la}, \psi \rangle_{2,\la} 
			\leq  C  \la^{-\frac{1}{4}+\frac{\theta}{2}}  \|\psi\|_{2,\la}.  \\
		\end{split}
	\end{equation*}
Then we conclude \eqref{l}.
	\end{proof}
	
\begin{lemma}\label{R_aaa} There is a constant $C>0$, independent of $\la$, such that
	\begin{equation*}\label{R}
		\|R_{\la}(\omega_1,\omega_2)\|_{\la} = O\Big(\la^{-1} \|(\omega_{1},\omega_{2})\|_{\la}^3 +\la^{-\frac{1}{2}}\|(\omega_{1},\omega_{2})\|_{\la}^2 
		\Big).
	\end{equation*}
Moreover, if $\|(\omega_{1},\omega_{2})\|_{\la}\leq C\|l_\la\|_\la$, then we have that
	\begin{equation*}
		\|R_{\la}((\omega_1,\omega_2))\|_{\la} = O(\la^{-1+\theta}). \end{equation*}
\end{lemma}

\begin{proof}
	For any $(\varphi,\psi) \in 
    H^1(\r^2)\times H^1(\r^2)$, we have 
	\[
	\langle R_{\la}(\omega_1,\omega_2), (\varphi,\psi) \rangle_\la
	=\langle R_{1,\la}(\omega_1,\omega_2),\varphi \rangle_{1,\la}
	+\langle R_{2,\la}(\omega_1,\omega_2), \psi \rangle_{2,\la}.
	\]
	By Sobolev embedding, it is easy to compute that, for any $q\geq2,\ i=1,2$,
	\begin{equation}\label{Holder}
		\begin{split}
			\|u\|_{L^{q}(\mathbb{R}^{2})} = {}& \la^{-\frac{1}{q}}\Bigl( \int_{\mathbb{R}^{2}} \big|u(\frac{y}{\sqrt{\la}})\big|^{q}dy \Bigr)^{\frac{1}{q}} \\
			\leq {}& C  \la^{-\frac{1}{q}}\Bigl(   \int_{\mathbb{R}^{2}} \Big[ \big|\nabla_{y}( u(\frac{y}{\sqrt{\la}})) \big|^{2} + u^{2}(\frac{y}{\sqrt{\la}})\Big] dy \Bigr)^{\frac{1}{2}}\\
			\leq {}& C\la^{-\frac{1}{q}} \|u\|_{i,\la}.
		\end{split}
	\end{equation}
	Since 
	\[
	R_{1,\lambda}(\omega_{1},\omega_{2}) = \alpha \Big(\omega_{1}^{3} + 3 U_{r_{\la},\la} \omega_{1}^{2}\Big) + \beta\Big(\omega_{1}\omega_{2}^{2} +  U_{r_{\la},\la} \omega_{2}^{2} + 2 V_{r_{\la},\la} \omega_{1}\omega_{2}\Big),
	\]
 by H\"older inequality, we have
	\[ 
	\begin{split}
		&\langle R_{1,\la}(\omega_1,\omega_2), \varphi \rangle_{1,\la} \\
		={}& \int_{0}^{+\infty} r \bigg[
		\alpha \Big(\omega_{1}^{3} + 3 U_{r_{\la},\la} \omega_{1}^{2}\Big) + \beta\Big(\omega_{1}\omega_{2}^{2} +  U_{r_{\la},\la} \omega_{2}^{2} + 2 V_{r_{\la},\la} \omega_{1}\omega_{2}\Big)
		\bigg]\varphi dr \\
		={}& \frac{1}{2\pi} \int_{\mathbb{R}^{2}}  \bigg[
		\alpha \Big(\omega_{1}^{3} + 3 U_{r_{\la},\la} \omega_{1}^{2}\Big) + \beta\Big(\omega_{1}\omega_{2}^{2} +  U_{r_{\la},\la} \omega_{2}^{2} + 2 V_{r_{\la},\la} \omega_{1}\omega_{2}\Big)
		\bigg]\varphi dx \\
		\leq{}& C \Big(\int_{\mathbb{R}^{2}} \omega_{1}^6 dx\Big)^{\frac{1}{2}} \Big(\int_{\mathbb{R}^{2}} \varphi^2 dx\Big)^{\frac{1}{2}}
		+ C \la^{\frac{1}{2}} \Big(\int_{\mathbb{R}^{2}} \omega_{1}^4 dx\Big)^{\frac{1}{2}} \Big(\int_{\mathbb{R}^{2}} \varphi^2 dx\Big)^{\frac{1}{2}} \\
		&+ C \Big(\int_{\mathbb{R}^{2}} \omega_{1}^6 dx\Big)^{\frac{1}{6}} \Big(\int_{\mathbb{R}^{2}} \omega_{2}^6 dx\Big)^{\frac{1}{3}} \Big(\int_{\mathbb{R}^{2}} \varphi^2 dx\Big)^{\frac{1}{2}} 
		+ C \la^{\frac{1}{2}} \Big(\int_{\mathbb{R}^{2}} \omega_{2}^4 dx\Big)^{\frac{1}{2}} \Big(\int_{\mathbb{R}^{2}} \varphi^2 dx\Big)^{\frac{1}{2}} \\
		& + C \la^{\frac{1}{2}} \Big(\int_{\mathbb{R}^{2}} \omega_{1}^3 dx\Big)^{\frac{1}{3}} \Big(\int_{\mathbb{R}^{2}} \omega_{2}^3 dx\Big)^{\frac{1}{3}} \Big(\int_{\mathbb{R}^{2}} \varphi^3 dx\Big)^{\frac{1}{3}} \\
		\leq {}& C\Big(\la^{-1} \|\omega_{1}\|_{1,\la}^3 +\la^{-\frac{1}{2}}\|\omega_{1}\|_{1,\la}^2 
		+ \la^{-1} \|\omega_{1}\|_{1,\la} \|\omega_{2}\|_{2,\la}^2
		+ \la^{-\frac{1}{2}}\|\omega_{2}\|_{2,\la}^2 
		+ \la^{-\frac{1}{2}}\|\omega_{1}\|_{1,\la} \|\omega_{2}\|_{2,\la} \Big)\|\varphi\|_{1,\la}  \\
		\leq {}& C\Big(\la^{-1} \|(\omega_{1},\omega_{2})\|_{\la}^3 +\la^{-\frac{1}{2}}\|(\omega_{1},\omega_{2})\|_{\la}^2 
		 \Big)\|\varphi\|_{1,\la}
	\end{split}
	\]
	Similarly, we have 
	\[ 
	\begin{split}
		\langle R_{2,\la}(\omega_1,\omega_2), \psi \rangle_{2,\la}
		\leq C\Big(\la^{-1} \|(\omega_{1},\omega_{2})\|_{\la}^3 +\la^{-\frac{1}{2}}\|(\omega_{1},\omega_{2})\|_{\la}^2 
		\Big)\|\psi\|_{2,\la}.
	\end{split}
	\]
	Thus, we complete the proof of Lemma \ref{R_aaa}.
\end{proof}

In order to apply the contraction mapping theorem, for $\theta$ small enough, we define the set
\[E:=\left\{(\omega_1,\omega_2):(\omega_1,\omega_2)\in E_{r_\lambda,\lambda}, \|(\omega_1,\omega_2)\|_\lambda\leq\lambda^{-\frac{1}{4}+\theta}\right\}.\]
From Proposition \ref{invertibility}, for any $(\omega_1,\omega_2)\in E_{r_\lambda,\lambda} $, we take
\[B(\omega_1,\omega_2):=L_\lambda^{-1}(l_\lambda+R_\lambda(\omega_1,\omega_2)).\]
\begin{proposition}\label{yasuo}
   For $\lambda$ large enough, there exist some small constant $\theta>0$ and  $(\omega_1,\omega_2)\in E_{r_\lambda,\lambda}$ satisfying \eqref{uv} and $\|(\omega_1,\omega_2)\|_{\la} =O \Bigl(\la^{-\frac{1}{4}+\frac{\theta}{2}}\Bigr)$.
\end{proposition}
\begin{proof}
    For any $(\omega_1,\omega_2)\in E$, combining Lemma \ref{l_aaa} and Lemma \ref{R_aaa}, we infer
    \begin{equation}
    \begin{split}
    \label{zishen}
\|B(\omega_1,\omega_2)\|_\lambda\leq {}& C\|l_\lambda\|_\lambda+C\|R_\lambda(\omega_1,\omega_2)\|_\lambda \\
\leq{}&C\lambda^{-\frac{1}{4}+\frac{\theta}{2}}+C(\lambda^{-1}\lambda^{-\frac{3}{4}+3\theta}+\lambda^{-\frac{1}{2}}\lambda^{-\frac{1}{2}+2\theta})\\
={}&C\lambda^{-\frac{1}{4}+\frac{\theta}{2}}+C(\lambda^{-\frac{7}{4}+3\theta}+\lambda^{-1+2\theta})\\
\leq{}&\lambda^{-\frac{1}{4}+\theta}.
    \end{split}    
    \end{equation}
On the other hand, for $(\omega_1,\omega_2), (\tilde{\omega_1},\tilde{\omega_2})\in E$, 
\[\|B(\omega_1,\omega_2)-B(\tilde{\omega_1},\tilde{\omega_2})\|_\lambda\leq C\|R_\lambda(\omega_1,\omega_2)-R_\lambda(\tilde{\omega_1},\tilde{\omega_2})\|_\lambda\leq\frac{1}{2}\|(\omega_1,\omega_2)-(\tilde{\omega_1},\tilde{\omega_2})\|_\lambda.\]
Thus, $B$ is a contraction mapping from $E$ to $E$ and there exists $(\omega_1,\omega_2)\in E$ such that $(\omega_1,\omega_2)=B(\omega_1,\omega_2)$. By \eqref{zishen}, we show
\[\|(\omega_1,\omega_2)\|_\lambda=\|B(\omega_1,\omega_2)\|_\lambda=O(\lambda^{-\frac{1}{4}+\frac{\theta}{2}}).\]
This completes the proof of Proposition \ref{yasuo}.
\end{proof}

As for $(\omega_1,\omega_2)$ obtained from Proposition \ref{yasuo}, we suppose that
\[u_\lambda(r):=U_{r_\lambda,\lambda}(r)+\omega_1(r),\]
\[v_\lambda(r):=V_{r_\lambda,\lambda}(r)+\omega_2(r).\]
Then, for some constants $b_{\lambda,1}$ and $b_{\lambda,2}$, $u_\lambda(r)$ and $v_\lambda(r)$ satisfy the following equations:
\begin{align}
-u''_\lambda(r) - \frac{1}{r}u'_\lambda(r) + \bigl(\lambda + P(r)\bigr)u_\lambda(r) - \alpha u_\lambda(r)^3 - \beta u _\lambda(r)v_\lambda(r)^2 &= b_{\lambda,1} U'_{r\lambda,\lambda}(r), \label{eq:system1} \\
-v''_\lambda(r) - \frac{1}{r}v'_\lambda(r) + \bigl(\lambda + Q(r)\bigr)v _\lambda(r)- \gamma v_\lambda(r)^3 - \beta u_\lambda(r)^2 v_\lambda(r) &= b_{\lambda,2} V'_{r\lambda,\lambda}(r). \label{eq:system2}
\end{align}

We will choose the parameter $r_\lambda$ appropriately to ensure that $b_{\lambda,1} = b_{\lambda,2} = 0$. 
\begin{proposition}
 If $r_\la$ satisfies the identity
\begin{align}
\begin{split}\label{=0}
&\int_0^{+\infty} \left( -u_\lambda'' - \frac{1}{r}u_\lambda' + \bigl(\lambda + P(r)\bigr)u_\lambda - \alpha u_\lambda^3 - \beta u_\lambda v_\lambda^2 \right) r^2 u_\lambda' \, dr \\
+&\int_0^{+\infty} \left( -v_\lambda'' - \frac{1}{r}v_\lambda' + \bigl(\lambda + Q(r)\bigr)v_\lambda - \gamma v_\lambda^3 - \beta u_\lambda^2 v_\lambda \right) r^2 v_\lambda' \, dr= 0,
\end{split}
\end{align}
then $b_{\lambda,1} = b_{\lambda,2} = 0$.   
\end{proposition}
\begin{proof}
    Considering \eqref{eq:system1}, \eqref{eq:system2} and \eqref{=0}, we conclude that 
    \begin{equation}
        \label{+=0}
b_{\lambda,1} \int_0^{+\infty} r^2 u_\lambda' U'_{r\lambda,\lambda} \, dr + b_{\lambda,2} \int_0^{+\infty} r^2 v_\lambda' V'_{r\lambda,\lambda} \, dr = 0.
\end{equation}

Indeed, we expand the integral as follows
\begin{align*}
\int_0^{+\infty} r^2 u_\lambda' U'_{r_\lambda,\lambda} \, dr &= \int_0^{+\infty} r^2 \left( U'_{r_\lambda,\lambda} + \omega_1' \right) U'_{r_\lambda,\lambda} \, dr \\
&= \int_0^{+\infty} r^2 U'^2_{r_\lambda,\lambda} \, dr + \int_0^{+\infty} r^2 \omega_1' U'_{r_\lambda,\lambda} \, dr.
\end{align*}
For the first term, from \eqref{rU'2}, we have the estimate
\[
\int_0^{+\infty} r^2 U'^2_{r_\lambda,\lambda} \, dr = A\frac{\gamma-\beta}{\alpha\gamma-\beta^2}r_\lambda^2\lambda^\frac{3}{2} + o\left( r_\lambda^2\lambda^\frac{3}{2} \right).
\]
For the second term, 
\begin{align*}
\int_0^{+\infty} r^2 \omega_1' U'_{r_\lambda,\lambda} \, dr= O\left(r_ \lambda \|w_1\|_{1,\lambda} \|U_{r_\lambda,\lambda}\|_{1,\lambda} \right) = O\left(r_\la^{\frac{3}{2}} \lambda^{\frac{1}{2} + \frac{\theta}{2}} \right),
\end{align*}
which implies
\begin{equation}
    \label{1xiang}
\int_0^{+\infty} r^2 u_\la' U'_{r_\lambda,\lambda} \, dr = A \frac{\gamma - \beta}{\alpha \gamma- \beta^2} r_\la^2\lambda^\frac{3}{2} + o\left(r_\la^2 \lambda^\frac{3}{2} \right) + O\left(r_\la^{\frac{3}{2}} \lambda^{\frac{1}{2} + \frac{\theta}{2}} \right).
\end{equation}
By the same argument, we obtain
\begin{equation}
    \label{2xiang}
\int_0^{+\infty} r^2 v_\la' V'_{r_\lambda,\lambda} \, dr = A \frac{\alpha - \beta}{\alpha\gamma - \beta^2} r_\la^2\lambda^\frac{3}{2} + o\left(r_\la^2 \lambda^\frac{3}{2} \right) + O\left(r_\la^{\frac{3}{2}} \lambda^{\frac{1}{2} + \frac{\theta}{2}} \right).
\end{equation}

Now, suppose that \(b_{\lambda,1}\) and \(b_{\lambda,2}\) are not both zero. We consider three cases.

Case 1: \(b_{\lambda,1} = 0\) and \(b_{\lambda,2} \neq 0\).
    From \eqref{+=0}, we have
    \[
    \int_0^{+\infty} r^2 v_\la' V'_{r_\lambda,\lambda} \, dr = 0,
    \]
    which is a contradiction since \eqref{2xiang}.
    
    Similarly, from \eqref{1xiang}-\eqref{2xiang}, we can derive that Case 2: \(b_{\lambda,1} \neq 0\), \(b_{\lambda,2} = 0\) and Case 3: \(b_{\lambda,1} \neq 0\), \(b_{\lambda,2} \neq 0\) are also contradictions.

Therefore, we can conclude that \(b_{\lambda,1} = b_{\lambda,2} = 0\).
\end{proof}

Next, we give some more precise estimates on $(u_{\la},v_{\la})$.
	
	\begin{lemma}\label{poho2estimate}
		Let $r_{\la}\in [\la^{1-\theta},\la^{1 + \theta}] $ and $\|\omega_{\la}\|_{\la} =O \Bigl(\la^{-\frac{1}{4}+\frac{\theta}{2}}\Bigr) $,
		where $\theta >0$ is a small fixed constant. Then it follows that
		\begin{equation}\label{2-30}
			\int_{0}^{+\infty} r \big[(\la+P(r))u_\la^2  + (\la+Q(r))v_\la^2 \big] dr 
			= 3\frac{\alpha+\gamma-2\beta}{\alpha\gamma-\beta^2}A \la^{\frac{3}{2}}r_{\la} + O\big(\la r_{\la} \big),
		\end{equation}
		\begin{equation}\label{2-31}
			\begin{split}
				\int_{0}^{+\infty} r (\alpha u_\la^4 + \gamma v_\la^4
				+2\beta  r u_\la^2 v_\la^2) dr 
				= 4\frac{\alpha+\gamma-2\beta}{\alpha\gamma-\beta^2}A \la^{\frac{3}{2}}r_{\la} + O\Big(\la^{1+\theta} \Big),
			\end{split}
		\end{equation}
		and
		\begin{equation}\label{2-32}
			\int_{0}^{+\infty} r^2 [P'(r)u_\la^2+Q'(r)v_\la^2] dr
			= \frac{3A}{\alpha\gamma-\beta^2} [(\gamma-\beta)P'(r_\la)+(\alpha-\beta)Q'(r_\la)] \la^{\frac{1}{2}}r_{\la}^2 + O\big(r_{\la}^2 \big).
		\end{equation}
	\end{lemma}
	
	\begin{proof}
		As for the term $\int_{0}^{+\infty}r u_{\la}^{4}d r$, direct computations give that
		\begin{align*}
			\int_{0}^{+\infty}r u_{\la}^{4} d r 
			=& \int_{0}^{+\infty}r U_{r_{\la},\la}^{4} d r + 4\int_{0}^{+\infty}r U_{r_{\la},\la}^{3} \omega_{1,\la}d r \\
			& +6\int_{0}^{+\infty}r U_{r_{\la},\la}^{2} \omega_{1,\la}^{2} d r + 4\int_{0}^{+\infty}r U_{r_{\la},\la}\omega_{1,\la}^{3}d r + \int_{0}^{+\infty}r \omega_{1,\la}^{4}d r .
		\end{align*}
		Since $U_{r_{\la},\la} \leq C\la^{\frac{1}{2}}$, then
		\[
		\int_{0}^{+\infty}r U_{r_{\la},\la}^{3} \omega_{1,\la} d r 
		= O\Big(\| U_{r_{\la},\la}\|_{1,\la}\cdot \| \omega_{1,\la} \|_{1,\la}\Big) 
		= O\Big(\la^{\frac{3}{4}}r_\la^{\frac{1}{2}} \| \omega_{1,\la} \|_{1,\la}\Big)
		= O\Big(\la^{1+\theta}\Big) ,
		\]
		\[
		\int_{0}^{+\infty}r U_{r_{\la},\la}^{2} \omega_{1,\la}^{2}d r = O\Big(\| \omega_{1,\la}\|_{1,\la}^{2}\Big)
		= O\Big(\la^{-\frac{1}{2}+\theta}\Big) .
		\]
		The H\"older inequality \eqref{Holder} gives that
		\[
		\int_{0}^{+\infty}r U_{r_{\la},\la} \omega_{1,\la}^{3}d r 
		\leq C \la^{\frac{1}{2}} 	\int_{0}^{+\infty}r  \omega_{1,\la}^{3}d r 
		= O\Big(  \la^{-\frac{1}{2}} \| \omega_{1,\la}\|_{1,\la}^{3}\Big)
		= O\Big(\la^{-\frac{5}{4}+\frac{3\theta}{2}}\Big) ,
		\]
		and 
		\[
		\int_{0}^{+\infty}r \omega_{1,\la}^{4} d r 
		=  O\Big( \la^{-1} \| \omega_{1,\la}\|_{1,\la}^{4}\Big)
		= O\Big(\la^{-2+2\theta}\Big).
		\]
		So we have
		\begin{equation}\label{2-33}
			\begin{split}
				\int_{0}^{+\infty}r u_{\la}^{4} d r 
				={}& \int_{0}^{+\infty}r U_{r_{\la},\la}^{4} d r + O\Big(\la^{1+\theta}\Big) \\
				= {}& r_\la \int_{0}^{+\infty} U_{r_{\la},\la}^{4} d r + O\Big(\la^{\frac{1}{2}}\Big) + O\Big(\la^{1+\theta}\Big) \\
				={}& \la^{\frac{3}{2}} \Big(1+\frac{P(r_\la)}{\la}\Big)^{\frac{3}{2}} r_\la \int_{-\sqrt{\la+P(r_\la)}r_\la}^{+\infty} U^4 dr + O\Big(\la^{1+\theta}\Big) \\
				={}& \Big(\frac{\gamma-\beta}{\alpha\gamma-\beta^2}\Big)^2 \la^{\frac{3}{2}} r_\la \int_{-\infty}^{+\infty} w^4 dr + O\Big(\la^{1+\theta}\Big) \\
				={}& 4 \Big(\frac{\gamma-\beta}{\alpha\gamma-\beta^2}\Big)^2 A \la^{\frac{3}{2}} r_\la  + O\Big(\la^{1+\theta}\Big),
			\end{split}
		\end{equation}
	where we used \eqref{pohozaev_single}.
		At the meantime, we can obtain that 
		\begin{equation}\label{2-34}
			\begin{split}
				\int_{0}^{+\infty}r v_{\la}^{4} d r 
				= 4 \Big(\frac{\alpha-\beta}{\alpha\gamma-\beta^2}\Big)^2 A \la^{\frac{3}{2}} r_\la  + O\Big(\la^{1+\theta}\Big),
			\end{split}
		\end{equation}
		and 
		\begin{equation}\label{2-35}
			\begin{split}
				\int_{0}^{+\infty}r u_{\la}^{2} v_{\la}^{2} d r 
				= 4 \frac{(\alpha-\beta)(\gamma-\beta)}{(\alpha\gamma-\beta^2)^2} A \la^{\frac{3}{2}} r_\la  + O\Big(\la^{1+\theta}\Big),
			\end{split}
		\end{equation}
		Hence \eqref{2-31} can be yielded by \eqref{2-33}, \eqref{2-34} and \eqref{2-35}.

		Similarly, we have
		\begin{equation}\label{2-36}
			\begin{split}
				\int_{0}^{+\infty} r \big[(\la+P(r))u_\la^2  + (\la+Q(r))v_\la^2 \big] dr  
				={}& \la r_\la \int_{0}^{+\infty}  \big( U_{r_{\la},\la}^2  + V_{r_{\la},\la}^2 \big) dr + O\big(\la r_{\la} \big) \\ 
				={}& \la^{\frac{3}{2}} r_\la \int_{-\infty}^{+\infty}  \big( U^2  + V^2 \big) dr + O\big(\la r_{\la} \big)  \\ 
				={}& 3\frac{\alpha+\gamma-2\beta}{\alpha\gamma-\beta^2}A \la^{\frac{3}{2}}r_{\la} + O\big(\la r_{\la} \big),
			\end{split}
		\end{equation}
		which yields \eqref{2-30}. Additionally, \eqref{2-32} holds by \eqref{2-11} and $\|\omega_{\la}\|_{\la} =O \Bigl(\la^{-\frac{1}{4}+\frac{\theta}{2}}\Bigr) $.
	\end{proof}

	\begin{proposition}
		If $r_{\la}\in [\la^{1-\theta},\la^{1 + \theta}]$, then \eqref{=0} is equivalent to
		\begin{equation}\label{M}
			M_{\la}'(r_{\la}) = O\big(\la^{- \frac{1}{2}}\big). 
		\end{equation}
	\end{proposition}
	
	\begin{proof}
		At first, we know \eqref{2-27} and \eqref{2-271} mean the solvability of problem \eqref{eqmain}. From the proof of Proposition \ref{poho}, we see \eqref{=0} is equivalent to \eqref{2-289}.  It is easy to check that, from \eqref{2-30}, \eqref{2-31} and \eqref{2-32}, it follows that
		\begin{equation*}
			\begin{split}
				\eqref{2-289} \Longleftrightarrow {}&
				1 +  \frac{3}{2} \Big(\frac{\gamma-\beta}{\alpha+\gamma-2\beta}  P'(r_\la)+ \frac{\alpha-\beta}{\alpha+\gamma-2\beta} Q'(r_\la) \Big) \frac{r_\la}{\la}  = O\big(\la^{- \frac{1}{2}}\big) \\
				\Longleftrightarrow {}& M_{\la}'(r_{\la}) = O\big(\la^{- \frac{1}{2}}\big). 
			\end{split}
		\end{equation*}
		Thus we can deduce that \eqref{=0} is equivalent to \eqref{M}.
	\end{proof}
	
\section{Normalized solutions}\label{cz2}

\begin{proof}[\textbf{Proof of Theorem \ref{diergedingli}.}] 
By \eqref{M1}, we observe that \eqref{M} implies $|r_\la-y_\la|=O(\la^{-\frac{1}{2}}).$ Therefore, there exists
\begin{equation}\label{rla}
    r_\la\in[y_\la-\la^{- \frac{1}{2}+\theta}, y_\la+ \la^{- \frac{1}{2}+\theta} ]
\end{equation} 
solving \eqref{=0}. The existence of solutions $(u_\la,v_\la)$ to system \eqref{eqmain_without} are established for $\la\geq\la_0$, where $\la_0$ is large enough. 

We now consider the $L^2$-constraint condition. Recall $\varepsilon= \frac{\alpha + \gamma - 2\beta}{\alpha\gamma - \beta^2}$. By \eqref{pohozaev_single} and \eqref{L2improve}, $\int_{\r^2}(u^2+v^2)dx = 1$ is equivalent to 
\begin{align*}
\frac{1}{2\pi\varepsilon} &= 3A  \lambda^{\frac{1}{2}} r_\lambda + o(\lambda^{\frac{1}{2}} r_\lambda).
\end{align*}

For any $\varepsilon\leq \va_0:=\frac{1}{8\pi A\la_0^{\frac{3}{2}+\theta}}$, we define the function \( F(\lambda) \) by
\[
F(\lambda) := 3A  \lambda^{\frac{1}{2}} r_\lambda - \frac{1}{2\pi\varepsilon} + o(\lambda^{\frac{1}{2}} r_\lambda) .
\]
 It is easy to check the following inequalities for \( F(\lambda) \),
\begin{align*}
F(\lambda) &\leq 3A  \lambda^{\frac{3}{2} + \theta}  - \frac{1}{2\pi\varepsilon} 
+ o\left(  \lambda^{\frac{3}{2}+\theta} \right), \\
F(\lambda) &\geq 3A \lambda^{\frac{3}{2} - \theta}- \frac{1}{2\pi\varepsilon}
+ o\left( \lambda^{\frac{3}{2}-\theta} \right) .
\end{align*}
To ensure \( F(\lambda_1) < 0 \), we choose
\[
\lambda_1 = \left( 8A  \pi \varepsilon\right)^{-\frac{1}{\frac{3}{2} + \theta}}.
\]
Then we have
\[
F(\la_1)\leq \la_1^{\frac{3}{2}+\theta}\left[3A-\frac{1}{2\pi\varepsilon\la_1^{\frac{3}{2}+\theta}}+o(1)\right]=\la_1^{\frac{3}{2}+\theta}\left(-A+o(1)\right)<0.
\]
By the mean time, we define
\[
\lambda_2 = \left( 4A  \pi \varepsilon \right)^{-\frac{1}{\frac{3}{2} - \theta}},
\]
which satisfies \( F(\lambda_2) > 0 \). It is clear that \(\la_0\leq \lambda_1 < \lambda_2 \) since $\theta>0$ is small.
 By the intermediate value theorem, there exists \( \lambda_\va \in (\lambda_1, \lambda_2) \) such that \( F(\lambda_\va) = 0 \). 
 
 Obviously, $\la_\va\to +\infty$ as $\va\to 0$. So if the function $M_\va$ satisfies \eqref{M1} and \eqref{M2}, then we will find a concentrated vector solution $(u_{\la_\va},v_{\la_\va})$ to problem $\eqref{eqmain}$ for any $\va\in(0,\va_0]$.
 For simplicity of notations, we rewrite $(u_{\la_\va},v_{\la_\va})$  as $(u_\va,v_\va)$ (see \eqref{tongshi1}).
 According to Proposition \ref{yasuo} and \eqref{rla}, we have
 \[
 \|\omega_\va\|_{\la_\va} = O\Big( \la_{\va}^{-\frac{1}{4}+\frac{\theta}{2}}\Big)
 \text{ and } 
|r_\va-y_\va|=O\Big(\la_{\va}^{-\frac{1}{2}+\theta} \Big). 
\]
Thus, the proof of Theorem \ref{diergedingli} is complete.
 \end{proof}
\medskip
\noindent\textbf{Acknowledgments}:  
 The authors are very grateful to Professor Qing Guo from Minzu University of China for the helpful discussion with her. This paper was supported by NSFC grants (No. 12526647).
 
%
%

\end{document}